\documentclass[reqno]{amsart}

\usepackage{hyperref}
\usepackage{bbm}
\usepackage{amsfonts,amsmath,amsxtra,amsthm,amssymb,latexsym}
\usepackage{ifpdf}
\usepackage{tikz}
\usepackage{color}

\newcommand*{\mailto}[1]{\href{mailto:#1}{\nolinkurl{#1}}}
\newcommand{\arxiv}[1]{\href{http://arxiv.org/abs/#1}{arXiv:#1}}
\newcommand{\doi}[1]{\href{http://dx.doi.org/#1}{doi: #1}}
\newcommand{\msc}[1]{\href{http://www.ams.org/msc/msc2010.html?t=&s=#1}{#1}}

\newcommand{\ack}{\section*{Acknowledgments}}
\newtheorem{theorem}{Theorem}[section]
\newtheorem{corollary}[theorem]{Corollary}
\newtheorem{lemma}[theorem]{Lemma}
\newtheorem{remark}[theorem]{Remark}
\newtheorem{hypothesis}{Hypothesis}[section]
\theoremstyle{definition}
\newtheorem{definition}[theorem]{Definition}
\newtheorem{example}[theorem]{Example}

\newcommand{\be}{\begin{equation}}
\newcommand{\ee}{\end{equation}}
\newcommand{\ti}{\tilde}
\newcommand{\id}{{\mathbbm 1}}

\numberwithin{equation}{section}

\DeclareMathOperator{\re}{Re} 
 
\DeclareMathOperator{\dom}{dom}
 
\DeclareMathOperator{\ess}{ess}

\DeclareMathOperator{\loc}{loc}

\newcommand\R{{\mathbb{R}}}
\newcommand\C{{\mathbb{C}}}
\newcommand\T{{\mathbb{T}}}
\newcommand\Z{{\mathbb{Z}}}
\newcommand\Q{{\mathbb{Q}}}

\newcommand{\gt}{\mathfrak{t}}

\newcommand\cA{{\mathcal{A}}}
\newcommand\cI{{\mathcal{I}}}
\newcommand\cT{{\mathcal{T}}}
\newcommand\cL{{\mathcal{L}}}
\newcommand\cK{{\mathcal{K}}}
\newcommand\cU{{\mathcal{U}}}
\newcommand\cG{{\mathcal{G}}}
\newcommand\cE{{\mathcal{E}}}
\newcommand\cP{{\mathcal{P}}}
\newcommand\cV{{\mathcal{V}}}
\newcommand\cY{{\mathcal{Y}}}
\newcommand\cC{{\mathcal{C}}}
\newcommand\OO{{\mathcal{O}}}

\newcommand\bH{{\mathbf{H}}}
\newcommand\rh{{\mathbf{h}}}
\newcommand\rH{{\rm{H}}}
\newcommand\E{{\rm{e}}}
\newcommand\mes{{\rm{mes}}}
\newcommand\vol{{\rm{vol}}}

\newcommand\rK{{\rm{K}}}
\newcommand\rD{{\rm{d}}}

\newcommand\comb{{\rm{comb}}}

\def\wt#1{{{\widetilde #1} }}

\newcommand{\floor}[1]{\lfloor #1 \rfloor}

\begin{document}

\title[Spectral Estimates for Quantum Graphs]{Spectral Estimates for Infinite Quantum Graphs}

\author[A. Kostenko]{Aleksey Kostenko}
\address{Faculty of Mathematics and Physics\\ University of Ljubljana\\ Jadranska ul.\ 21\\ 1000 Ljubljana\\ Slovenia\\ and Faculty of Mathematics\\ University of Vienna\\ 
Oskar-Morgenstern-Platz 1\\ 1090 Vienna\\ Austria}
\email{\mailto{Aleksey.Kostenko@fmf.uni-lj.si};\ \mailto{Oleksiy.Kostenko@univie.ac.at}}
\urladdr{\url{http://www.mat.univie.ac.at/~kostenko/}}

\author[N. Nicolussi]{Noema Nicolussi}
\address{Faculty of Mathematics\\ University of Vienna\\
Oskar-Morgenstern-Platz 1\\ 1090 Vienna\\ Austria}
\email{\mailto{noema.nicolussi@univie.ac.at}}

\thanks{{\it Research supported by the Austrian Science Fund (FWF) 
under Grants No.\ P 28807 (A.K. and N.N.) and W 1245 (N.N.).}}
\thanks{{\it Calc.\ Var.\ \& Partial Differential Equations}, {\bf 58}:15 (2019)}
\thanks{\doi{10.1007/s00526-018-1454-3}}
\thanks{\arxiv{1711.02428}}


\begin{abstract}
We investigate the bottom of the spectra of infinite quantum graphs, i.e., Laplace operators on metric graphs having infinitely many edges and vertices. We introduce a new definition of the isoperimetric constant for quantum graphs and then prove the Cheeger-type estimate. Our definition of the isoperimetric constant is purely combinatorial and thus it establishes connections with the combinatorial isoperimetric constant, one of the central objects in spectral graph theory and in the theory of simple random walks on graphs. The latter enables us to prove a number of criteria for quantum graphs to be uniformly positive or to have purely discrete spectrum. We demonstrate our findings by considering trees, antitrees and Cayley graphs of finitely generated groups.
\end{abstract}

\maketitle

{\scriptsize{\tableofcontents}}
%%%%%%%%%%%%%%%%%%%%%%%%%%%%%%%%%%%%%%%%%%%%%%%%%%%%%%%%%%%%
%%%%%%%%%%%%%%%%%%%%%%%%%%%%%%%%%%%%%%%%%%%%%%%%%%%%%%%%%%%%
\section{Introduction}
%%%%%%%%%%%%%%%%%%%%%%%%%%%%%%%%%%%%%%%%%%%%%%%%%%%%%%%%%%%%
%%%%%%%%%%%%%%%%%%%%%%%%%%%%%%%%%%%%%%%%%%%%%%%%%%%%%%%%%%%%

The main focus of our paper is on the study of spectra of quantum graphs. The notion of ``quantum graph" refers to a graph $\cG$ considered as a one-dimensional simplicial complex and equipped with a differential operator. The
spectral and scattering properties of Schr\"odinger operators on such structures attracted a considerable interest during the last two decades, as they provide, in particular, relevant models of nanostructured systems (we only mention recent collected works and monographs with a comprehensive bibliography:  \cite{bcfk06},  \cite{bk13}, \cite{ekkst08}, \cite{post}).

Let $\cG$ be a locally finite connected metric graph, that is, a locally finite connected combinatorial graph $\cG_d=(\cV,\cE)$, where each edge $e\in \cE$ is identified with a copy of the interval $[0,|e|]$ and $|\cdot|$ denotes the edge length. We shall always assume throughout the paper that {\em each edge has finite length}, that is, $|\cdot|\colon \cE\to (0,\infty)$. In the Hilbert space $L^2(\cG) = \bigoplus_{e\in\cE} L^2(e)$, we can define the Hamiltonian $\bH$ which acts in this space as the (negative) second derivative $-\frac{d^2}{dx_e^2}$ on every edge $e\in\cE$. To give $\bH$ the meaning of a quantum mechanical energy operator, it must be self-adjoint and hence one needs to impose appropriate boundary conditions at the vertices.  Kirchhoff (also known as Kirchhoff--Neumann) conditions \eqref{eq:kirchhoff} are the most standard ones (cf. \cite{bk13}) and  
the corresponding operator denoted by $\bH$ is usually called a Kirchhoff (Kirchhoff--Neumann) Laplacian (we refer to Sections \ref{ss:II.02}--\ref{ss:III.01} for a precise definition of the operator $\bH$). If the graph $\cG$ is finite ($\cG$ has finitely many vertices and edges), then the spectrum of $\bH$ is purely discrete (see, e.g., \cite{bk13}). During the last few years, a lot of effort has been put in estimating the first nonzero eigenvalue of the operator $\bH$ (notice that $0$ is always a simple eigenvalue if $\cG_d$ is connected) and also in understanding its dependence on various characteristics of the corresponding metric graph including the number of {\em essential} vertices of the graph (vertices of degree $2$ are called {\em inessential}); the number or the total length of the graph's edges;  the edge connectivity of the underlying (combinatorial) graph, etc. 
For further information we refer to a brief selection of recent articles \cite{ari}, \cite{bale}, \cite{bkkm}, \cite{km16}, \cite{kkmm16}, \cite{kn14}, \cite{roh}.  

If the graph $\cG$ is infinite (there are infinitely many vertices and edges), then the corresponding pre-minimal operator $\bH_0$ defined by \eqref{eq:H0} is not automatically essentially self-adjoint. One of the standard conditions to ensure the essential self-adjointness of $\bH_0$ is the existence of a positive lower bound on the edges lengths, $\ell_\ast(\cG) = \inf_{e\in\cE}|e|>0$ (see \cite{bk13}). Only recently several self-adjointness conditions without this rather restrictive assumption have been established in \cite{ekmn}, \cite{kmne17} 
 (see Section \ref{ss:II.03} for further details). Of course, the next natural question is the structure of the spectrum of the operator $\bH$. Clearly, the spectrum of an infinite quantum graph is not necessarily discrete and hence one is interested in the location of the bottom of the spectrum, $\lambda_0(\bH)$, as well as of the bottom of the essential spectrum, $\lambda_0^{\ess}(\bH)$, of $\bH$. Since the graph is infinite, many quantities of interest for finite quantum graphs (e.g., the number of vertices, edges, or its total length) are no longer suitable for these purposes and the corresponding bounds usually lead to trivial estimates. However, it is widely known that quantum graphs in a certain sense interpolate between Laplacians on Riemannian manifolds and difference Laplacians on combinatorial graphs and hence quantum graphs can be investigated by modifying techniques that have been developed for operators on manifolds and graphs and we explore these analogies in the present paper. Notice that this insight has already proved to be very fruitful and it has led to many important results in spectral theory of operators on metric graphs (see, e.g., \cite{bk13}). Although quantum graphs are essentially operators on one-dimensional manifolds, our point of view is that the corresponding results and estimates should be of combinatorial nature.

Our central result is a Cheeger-type estimate for quantum graphs, which establishes lower bounds for $\lambda_0(\bH)$ and $\lambda_0^{\ess}(\bH)$ in terms of the isoperimetric constant $\alpha(\cG)$ of the metric graph $\cG$ (Theorem \ref{th:Cheeger}). Although the Cheeger-type bound for (finite) quantum graphs was proved 30 years ago by S.\ Nicaise (see \cite[Theorem 3.2]{nic}), we give a new purely combinatorial definition of the isoperimetric constant (see Definition \ref{def:iso}) and as a result this establishes a connection with isoperimetric constants for combinatorial graphs (see Lemma \ref{lem:Cheegerconnection} and also \eqref{eq:A=Acomb01}--\eqref{eq:A=Acomb02}). To a certain extent this connection is expected (cf. Theorem \ref{th:below} and also \cite{vB,cat,bgp,post}). Moreover, it was observed recently in \cite{ekmn,kmne17} by using the ideas from \cite{KM10} that spectral properties of the operator $\bH$ are closely connected with the corresponding properties of the discrete Laplacian defined in $\ell^2(\cV;m)$ by the expression
\be\label{eq:1.03}
(\tau_{\cG} f)(v) := \frac{1}{m(v)} \sum_{u\sim v} \frac{f(v) - f(u)}{|e_{u,v}|},\quad v\in\cV,
\ee
where the weight function $m\colon \cV\to \R_{>0}$ is given by 
\be\label{eq:1.04}
m\colon v\mapsto \sum_{u\sim v}|e_{u,v}|. 
\ee
Using this connection, several criteria for $\lambda_0(\bH)$ and $\lambda_0^{\ess}(\bH)$ to be positive have been established in \cite{ekmn}, however, in terms of isoperimetric constants and volume growth of the combinatorial graphs, which were introduced, respectively, in \cite{bkw15} and \cite{fo}, \cite{hkw13} (in this paper we obtain these results as simple corollaries of our estimate \eqref{eq:Cheeger}).
 
 Despite the combinatorial nature of \eqref{eq:alphaG} and \eqref{eq:alphaessG}, it is known that computation of the combinatorial isoperimetric constant is an NP-hard problem \cite{moh89} (see also \cite{hoo,kai} for further details). Motivated by \cite{bkw15} and \cite{dk}, we introduce a quantity, which sometimes is interpreted as a curvature of a graph, leading to estimates for the isoperimetric constants $\alpha(\cG)$ and $\alpha_{\ess}(\cG)$. It also turns out to be very useful in many situations of interest as we show by the examples of trees and antitrees. Another way to estimate isoperimetric constants is provided by the volume growth. Namely, we can apply the exponential volume growth estimates for regular Dirichlet forms from \cite{stu} (see also \cite{hkw13}, \cite{not}) to prove upper bounds (Brooks-type estimates \cite{bro}) for quantum graphs (see Theorem \ref{th:brooks}). However, this can be done under the additional assumption that the metric graph is complete with respect to the natural path metric (notice that in this case $\bH_0$ is essentially self-adjoint and $\bH$ coincides with its closure, see Corollary \ref{cor:gaffney}).

The quantities $\lambda_0(\bH)$ and $\lambda_0^{\ess}(\bH)$ are of fundamental importance for several reasons. From the spectral theory point of view, the positivity of $\lambda_0(\bH)$ or $\lambda_0^{\ess}(\bH)$ corresponds to bounded invertibility or Fredholmness of the operator $\bH$. Moreover, $\lambda_0^{\ess}(\bH)=+\infty$ holds precisely when the spectrum of $\bH$ is purely discrete, which is further equivalent to the compactness of the embedding $H^1_0(\cG)$ into $L^2(\cG)$ (the definition of the form domain $H^1_0(\cG)$ is given in Section \ref{ss:III.01}). It is difficult to overestimate the importance of  $\lambda_0(\bH)$ and $\lambda_0^{\ess}(\bH)$ in applications. For example, in the theory of parabolic equations $\lambda_0(\bH)$ gives the speed of convergence of the system towards equilibrium. On the other hand, Cheeger-type inequalities have a venerable history. Starting from the seminal work of J.\ Cheeger \cite{che}, where a connection between the isoperimetric constant of a compact manifold and a first nontrivial eigenvalue of the Laplace--Beltrami operator was found, this topic became an active area of research in both manifolds and graphs settings. One of the most fruitful applications of Cheeger's inequality in graph theory (this inequality was first proved independently in \cite{dod,dk86} and \cite{al, am}) is in the study of networks connectivity, namely, in constructing expanders (see \cite{cdv,dsv,hoo,lub}). Notice also that the positivity of the isoperimetric constant (also known as a {\em strong isoperimetric inequality}) is of fundamental importance in the study of random walks on graphs (we refer to \cite{woe} for further details). 

Let us now finish the introduction by describing the content of the article. First of all, we review necessary notions and facts on infinite quantum graphs in Section~\ref{sec:QG}, where we introduce the pre-minimal operator $\bH_0$ (Section~\ref{ss:II.02}), discuss its essential self-adjointness (Section~\ref{ss:II.03}) and the corresponding quadratic form $\gt_\cG$ (Section~\ref{ss:III.01}), and also touch upon its connection with the difference Laplacian~\eqref{eq:1.03} (Section~\ref{ss:III.00}). 

Section~\ref{ss:III.02} contains our first main result, Theorem~\ref{th:Cheeger}, which provides the Cheeger-type estimate for quantum graphs. Its proof follows closely the line of arguments as in the manifold case with the only exception, Lemma~\ref{lem:chopensets}, which enables us to replace the isoperimetric constant \eqref{eq:alphaG02} having the form similar to that of in \cite{nic} (see also \cite{km16, post09}) by the quantity \eqref{eq:alphaG} having a combinatorial structure. The latter also reveals connections with the combinatorial isoperimetric constant $\alpha_{\comb}$ from \cite{am,dod}, which measures connectedness of the underlying combinatorial graph, and with the discrete isoperimetric constant $\alpha_d$ introduced recently in \cite{bkw15} for the difference Laplacian \eqref{eq:1.03}. Bearing in mind the importance of both $\alpha_{\comb}$ and $\alpha_d$ in applications as well as the fact that these quantities are widely studied, we discuss these connections in Sections~\ref{ss:III.03}. 

Similar to manifolds and combinatorial Laplacians, one can estimate $\lambda_0(\bH)$ and $\lambda_0^{\ess}(\bH)$ by using the isoperimetric constant not only from below but also from above (Lemma \ref{lem:est03}). However, the price we have to pay is the existence of a positive lower bound on the edges lengths, $\inf_{e\in\cE}|e|>0$. Combining these estimates with the results from Section~\ref{ss:III.03}, we conclude that in this case the positivity of $\lambda_0(\bH)$ (resp., $\lambda_0^{\ess}(\bH)$) is equivalent to the validity of a strong isoperimetric inequality, i.e., $\alpha_{\comb}>0$ (resp., $\alpha_{\comb}^{\ess}>0$). 

In Section~\ref{ss:III.06}, we introduce a quantity which may be interpreted as a curvature of a metric graph. Firstly, using this quantity we are able to obtain estimates on the isoperimetric constant. Secondly, we discuss its connection with the curvatures introduced for combinatorial Laplacians in \cite{dk} and for unbounded difference Laplacians in \cite{bkw15}. The latter, in particular, enables us to obtain simple discreteness criteria for $\sigma(\bH)$ (see Lemma \ref{lem:curv_combess} and Corollary \ref{cor:curv_combess}), which to a certain extent can be seen as the analogs of the discreteness criteria from \cite{dl} and \cite{fuj}. 

The estimates in terms of the volume growth are given in Section~\ref{ss:III.05}. 
In Section~\ref{sec:Examples}, we consider several illustrative examples. The case of trees is treated in Section~\ref{ss:IV.02}. We show that for trees without inessential vertices and loose ends (vertices having degree $1$), $\lambda_0(\bH)>0$ if and only if $\sup_{e}|e|<\infty$. Moreover, the spectrum of $\bH$ is purely discrete if and only if the number $\#\{e\in\cE\colon |e|>\varepsilon\}$ is finite for every $\varepsilon>0$. Notice that under the additional symmetry assumption that a given metric tree is regular similar results, however, for the so-called Neumann Laplacian were observed by M.\ Solomyak \cite{sol04}. The case of antitrees is considered in Section \ref{ss:IV.04}. We provide some general estimates and also focus on two particular examples of exponentially and polynomially growing antitrees. In particular, it turns out that for a polynomially growing antitree, our results provide rather good estimates for $\lambda_0(\bH)$ and $\lambda_0^{\ess}(\bH)$ (see Example \ref{lem:at_qs}). In the last subsection, we consider the case of Cayley graphs of finitely generated groups. Similar to combinatorial Laplacians, the amenability/non-amenability of the underlying group plays a crucial role. 

Finally, in Appendix~\ref{sec:app} we provide a slight improvement to the Cheeger estimates from \cite{bkw15} by noting that one can replace intrinsic path metrics in the definition of isoperimetric constants simply by edge weight functions having {\em an intrinsic} property.

%%%%%%%%%%%%%%%%%%%%%%%%%%%%%%%%%%%%%%%%%%%%%%%%%%%%%%%%%%%%
%%%%%%%%%%%%%%%%%%%%%%%%%%%%%%%%%%%%%%%%%%%%%%%%%%%%%%%%%%%%
\section{Quantum graphs}\label{sec:QG}
%%%%%%%%%%%%%%%%%%%%%%%%%%%%%%%%%%%%%%%%%%%%%%%%%%%%%%%%%%%%
%%%%%%%%%%%%%%%%%%%%%%%%%%%%%%%%%%%%%%%%%%%%%%%%%%%%%%%%%%%%
%%%%%%%%%%%%%%%%%%%%%%%%%%%%%%%%%%%%%%%%%%%%%%%%%%%%%%%%%%%%
\subsection{Combinatorial and metric graphs}\label{ss:II.01}
%%%%%%%%%%%%%%%%%%%%%%%%%%%%%%%%%%%%%%%%%%%%%%%%%%%%%%%%%%%%
In what follows, $\cG_d = (\cV, \cE)$ will be an unoriented graph with countably infinite sets of vertices $\cV$ and edges $\cE$. 
For two vertices $u$, $v\in \cV$ we shall write $u\sim v$ if there is an edge $e_{u,v}\in \cE$ connecting $u$ with $v$. 
For every $v\in \cV$, we denote  the set of edges incident to the vertex $v$ by $\cE_v$ and 
\be\label{eq:combdeg}
\deg_{\cG}(v):= \#\{e|\, e\in\cE_v\} 
\ee
is called {\em the degree} (or {\em combinatorial degree}) of a vertex $v\in\cV$. When there is no risk of confusion which graph is involved, we shall write $\deg$ instead of $\deg_{\cG}$. By $\#(S)$ we denote the cardinality of a given set $S$. 
{\em A path} $\cP$ of length $n\in\Z_{>0}\cup\{\infty\}$ is a sequence of vertices $\{v_0,v_1,\dots, v_n\}$ such that $v_{k-1}\sim v_k$ for all $k\in \{1,\dots,n\}$. If $v_0 = v_n$, then $\cP$ is called a {\em cycle}.

We shall always make the following assumption.

\begin{hypothesis}\label{hyp:locfin}
The infinite graph $\cG_d$ is {\em locally finite} ($\deg(v) < \infty$ for every $v \in \cV$), {\em connected} (for any two vertices $u,v\in\cV$  there is a path connecting $u$ and $v$), and {\em simple} (there are no loops or multiple edges).
\end{hypothesis}

Next we assign each edge $e \in \cE$ a finite length $|e|\in (0,\infty)$. In this case $\cG:=(\cV,\cE,|\cdot|) = (\cG_d,|\cdot|)$ is called {\em a metric graph}. The latter enables us to equip $\cG$ with a topology and metric. Namely, by assigning each edge a direction and calling one of its vertices the initial vertex $e_0$ and the other one the terminal vertex $e_{i}$, every edge $e\in\cE$ can be identified with a copy of the interval $\cI_e= [0,|e|]$; moreover, the ends of the edges that correspond to the same vertex $v$ are identified as well. Thus, $\cG$ can be equipped with {\em the natural path metric} $\varrho_0$ (the distance between two points $x,y\in\cG$ is defined as the length of the ``shortest" path connecting $x$ and $y$). Moreover, a metric graph $\cG$ can be considered as a topological space (one-dimensional simplicial complex). For further details we refer to, e.g., \cite[Chapter 1.3]{bk13}. 

Also throughout this paper we shall assume the following conditions.

\begin{hypothesis}\label{hyp:01}
There is a finite upper bound for lengths of graph edges:
\be\label{eq:hyp1}
\ell^\ast(\cG):= \sup_{e\in\cE} |e|<\infty.
\ee 
\end{hypothesis}

In fact, Hypothesis \ref{hyp:01} is not a restriction for our purposes (see Lemma \ref{lem:est01} and also Remark \ref{rem:3.3}(i)).

\begin{hypothesis}\label{hyp:02}
All edges in $\cG$ are essential, that is, $\deg(v)\neq 2$ for all $v\in\cV$.
\end{hypothesis}

This assumption is not a restriction at  all since vertices of degree $2$ are irrelevant for the spectral properties of the Kirchhoff  Laplacian and hence can be removed (see, e.g., \cite{km16}).

%%%%%%%%%%%%%%%%%%%%%%%%%%%%%%%%%%%%%%%%%%%%%%%%%%%%%%%%%%%%
\subsection{Kirchhoff's Laplacian}\label{ss:II.02}
%%%%%%%%%%%%%%%%%%%%%%%%%%%%%%%%%%%%%%%%%%%%%%%%%%%%%%%%%%%%

Let $\cG$ be a metric graph satisfying Hypothesis \ref{hyp:locfin}--\ref{hyp:02}. Upon identifying every $e\in\cE$ with a copy of the interval $\cI_e$  and considering $\cG$ as the union of all edges glued together at certain endpoints, let us introduce the Hilbert space $L^2(\cG)$ of functions $f\colon \cG\to \C$ such that 
\[
L^2(\cG) = \bigoplus_{e\in\cE} L^2(e) = \Big\{f=\{f_e\}_{e\in\cE}\big|\, f_e\in L^2(e),\ \sum_{e\in\cE}\|f_e\|^2_{L^2(e)}<\infty\Big\}.
\]
The subspace of compactly supported $L^2(\cG)$ functions will be denoted by
\begin{equation*}
	L^2_c(\cG) = \big\{f \in L^2(\cG)| \; f \neq 0 \text{ only on finitely many edges } e \in \cE\big\}.
\end{equation*}

Next let us equip  $\cG$ with the Laplace operator. For every $e\in\cE$ consider the maximal operator $\rH_{e,\max}$ acting on functions $f\in H^2(e)$ as a negative second derivative. Here and below $H^n(e)$ for $n\in \Z_{\ge 0}$ denotes the usual Sobolev space. In particular, $H^0(e)= L^2(e)$ and 
\[
H^1(e) = \{f\in AC(e)\colon f'\in L^2(e)\},\quad H^2(e) = \{f\in H^1(e)\colon f'\in H^1(e)\}.
\]
%, which consists of functions which are absolutely continuous on $e$ together with their first derivatives and such that their second derivative belongs to $L^2(e)$. 
Now consider the maximal operator on $\cG$ defined by
\be\label{eq:Hmax}
\bH_{\max} = \bigoplus_{e\in \cE} \rH_{e,\max},\qquad \rH_{e,\max} = -\frac{\rD^2}{\rD x_e^2},\quad \dom(\rH_{e,\max}) = H^2(e).
\ee
For every $f_e\in H^2(e)$ the following quantities 
\begin{align}\label{eq:tr_fe}
f_e(e_o) & := \lim_{x\to e_o} f_e(x), &  f_e(e_i) & := \lim_{x\to e_i} f_e(x),
\end{align}
and 
\begin{align}\label{eq:tr_fe'}
f_e'(e_o) & := \lim_{x\to e_o} \frac{f_e(x) - f_e(e_o)}{|x - e_o|}, & f_e'(e_i) & := \lim_{x\to e_i} \frac{f_e(x) - f_e(e_i)}{|x - e_i|},
\end{align}
are well defined. The Kirchhoff (or Kirchhoff--Neumann) boundary conditions at every vertex $v\in\cV$ are then given by
\be\label{eq:kirchhoff}
\begin{cases} f\ \text{is continuous at}\ v,\\[2mm] \sum_{e\in \cE_v}f_e'(v) =0. \end{cases}
\ee 

Imposing these boundary conditions on the maximal domain $\dom(\bH_{\max})$ and then restricting to compactly supported functions we get the pre-minimal operator
\be\label{eq:H0}
\begin{split}
	\bH_{0} & =  \bH_{\max}\upharpoonright {\dom(\bH_{0})},\\ 
	& \dom(\bH_{0}) = \{f\in \dom(\bH_{\max})\cap L^2_{c}(\cG)|\, f\ \text{satisfies}\ \eqref{eq:kirchhoff},\ v\in\cV\}.
\end{split}
\ee
Integrating by parts one obtains that $\bH_0$ is symmetric. We call its closure {\em the minimal Kirchhoff Laplacian}. Notice that the values of $f$ at the vertices \eqref{eq:tr_fe} and one-sided derivatives \eqref{eq:tr_fe'} do not depend on the choice of orientation on $\cG$. Moreover, the second derivative is also independent of orientation on $\cG$ and hence so is the operator $\bH_0$.

\begin{remark}\label{rem:BC}
If $\deg(v) = 1$, then Kirchhoff's condition \eqref{eq:kirchhoff} at $v$ is simply the Neumann condition
\be\label{eq:neumann}
f_e'(v) = 0.
\ee
Let us mention that one can replace it by the Dirichlet condition
\be\label{eq:dirichlet}
f_e(v) = 0
\ee
and we shall consider the operator $\bH_0$ with mixed boundary conditions (either Neumann or Dirichlet) at the vertices $v\in\cV$ of the graph $\cG$ such that $\deg(v)=1$. 
\end{remark}

In the rest of our paper, we shall denote  by $\cV_D$ (respectively, by $\cV_N$) the set of vertices $v\in\cV$ such that $\deg(v)=1$ and the Dirichlet condition \eqref{eq:dirichlet} (respectively, the Neumann condition \eqref{eq:neumann}) is imposed at $v$. The sets of corresponding edges will be denoted by $\cE_D$ and $\cE_N$, respectively.

%%%%%%%%%%%%%%%%%%%%%%%%%%%%%%%%%%%%%%%%%%%%%%%%%%%%%%%%%%%%
\subsection{Self-adjointness}\label{ss:II.03}
%%%%%%%%%%%%%%%%%%%%%%%%%%%%%%%%%%%%%%%%%%%%%%%%%%%%%%%%%%%%

In the rest of our paper we shall always assume that the graph $\cG_d$ is infinite, that is, both sets $\cV$ and $\cE$ are infinite  (since $\cG_d$ is assumed to be locally finite).
In this case the operator $\bH_0$ is not necessarily essentially self-adjoint (that is, its closure may have nonzero deficiency indices) and finding self-adjointness criteria is a challenging open problem. The next results were proved recently in \cite{ekmn}. 
 Define the weight function $m\colon\cV\to \R_{>0}$ by 
\be\label{eq:m_def}
m\colon v\mapsto \sum_{e\in\cE_v}|e|, 
\ee
and then let $p_m\colon \cE\to \R_{>0}$ be given by 
\be\label{eq:p_def}
p_m\colon e_{u,v}\mapsto m(u) + m(v). 
\ee
The path metric $\varrho_m$ on $\cV$ generated by $p_m$ is defined by
\be\label{eq:pathmetric_m}
\varrho_m(u,v) := 
\inf_{\cP=\{v_0,\dots,v_n\}\colon v_0=u\ v_n=v}\sum_{k} p_m(e_{v_{k-1},v_k}),  
\ee
where the infimum is taken over all paths connecting $u$ and $v$.

\begin{theorem}[\cite{ekmn}]\label{th:sa}
If $(\cV,\varrho_m)$ is complete as a metric space, then $\bH_0$ is essentially self-adjoint. 
In particular, $\bH_0$ is essentially self-adjoint if
\be\label{eq:mast}
\inf_{v\in\cV} m(v) >0.
\ee
\end{theorem}

Replacing $p_m$ in \eqref{eq:pathmetric_m} by the edge length $|\cdot|$, we end up with the natural path metric $\varrho_0$ on $\cV$. Clearly, $(\cV,\varrho_m)$ is complete if so is $(\cV,\varrho_0)$ and hence we arrive at the following Gaffney-type theorem for quantum graphs.

\begin{corollary}[\cite{ekmn}]\label{cor:gaffney} 
If $\cG$ equipped with a natural path metric is complete as a metric space, then $\bH_0$ is essentially self-adjoint.
\end{corollary}

The next well known result (see \cite[Theorem 1.4.19]{bk13}) also immediately follows from Theorem \ref{th:sa}.
\begin{corollary}\label{cor:inf} 
If 
\be\label{eq:ell>0}
\ell_\ast(\cG) := \inf_{e\in\cE} |e|>0,
\ee
then $\bH_0$ is essentially self-adjoint.
\end{corollary}

%%%%%%%%%%%%%%%%%%%%%%%%%%%%%%%%%%%%%%%%%%%%%%%%%%%%%%%%%%%%
\subsection{Quadratic forms}\label{ss:III.01}
%%%%%%%%%%%%%%%%%%%%%%%%%%%%%%%%%%%%%%%%%%%%%%%%%%%%%%%%%%%%

In this section we present the variational definition of the Kirchhoff Laplacian. Consider the quadratic form
\be\label{eq:gt0}
\gt_\cG^0[f]:= (\bH_0 f,f)_{L^2(\cG)},\qquad f\in \dom(\gt^0_\cG):=\dom(\bH_0).
\ee
For every $f\in \dom(\bH_0)$, an integration by parts gives
\be\label{eq:gt00}
\gt^0_\cG[f] = \int_\cG |f'(x)|^2\,dx = \|f'\|^2_{L^2(\cG)}.
\ee
Clearly, the form $\gt_\cG^0$ is nonnegative. Moreover, it is closable since $\bH_0$ is symmetric. Let us denote its closure by $\gt_\cG$ and the corresponding domain by $H^1_0(\cG):=\dom(\gt_\cG)$. By the first representation theorem, there is a unique nonnegative self-adjoint operator corresponding to the form $\gt_\cG$.

\begin{definition}\label{def:KLapl}
The self-adjoint nonnegative operator $\bH$ associated with the form $\gt_\cG$ in $L^2(\cG)$ will be called {\em the Kirchhoff Laplacian}.
\end{definition}

If the pre-minimal operator $\bH_0$ is essentially self-adjoint, then $\bH$ coincides with its closure. In the case when $\bH_0$ is a symmetric operator with nontrivial deficiency indices, the operator $\bH$ is the Friedrichs extension of $\bH_0$.  

\begin{remark}\label{rem:maxform}
Of course, one may consider the maximally defined form
\be
\gt_\cG^{(N)}[f] :=   \int_\cG |f'(x)|^2\,dx,\qquad f\in \dom(\gt^{(N)}_\cG),
\ee
where
\be
\dom(\gt^{(N)}_\cG):= \{f\in L^2(\cG)|\ f\in H^1_{\loc}(\cG),\ f'\in L^2(\cG)\} = :H^1(\cG),
\ee
and then associate a self-adjoint positive operator, let us denote it by $\bH^N$, with this form in $L^2(\cG)$.  
Clearly, the forms $\gt_\cG$ and $\gt_\cG^{(N)}$ coincide if and only if $\bH$ is the unique positive self-adjoint extension of $\bH_0$ (this in particular holds if $\bH_0$ is essentially self-adjoint). We are not aware of a description of the self-adjoint operator $\bH^N$ associated with the form $\gt_\cG^{(N)}$ if the pre-minimal operator has nontrivial deficiency indices (however, see the recent work \cite{car11,klss}). Moreover, to the best of our knowledge, the description of deficiency indices of $\bH_0$ and its self-adjoint extensions is a widely open problem.  
\end{remark}

If at some vertices $v\in\cV$ with $\deg(v)=1$ the Neumann condition \eqref{eq:neumann} is replaced by the Dirichlet condition \eqref{eq:dirichlet}, then the corresponding form domain will be denoted by $\wt{H}^1_0(\cG)$. Notice that 
\be\label{eq:DN}
\wt{H}^1_0(\cG) = \{f\in H^1_0(\cG)|\ f_e(v)=0,\ v\in \cV_D\}.
\ee
By abusing the notation, we shall denote the corresponding self-adjoint operator by $\bH$. 
The bottom of the spectrum of $\bH$ can be found by using the Rayleigh quotient 
\be\label{eq:Rayleigh}
\lambda_0(\bH):=\inf \sigma(\bH) = \inf_{\substack{f\in \wt H^1_0(\cG)\\ f\neq0}}\frac{(\bH f,f)_{L^2(\cG)}}{\|f\|^2_{L^2(\cG)}} = \inf_{\substack{f\in \wt H^1_0(\cG)\\ f\neq0}}\frac{\|f'\|^2_{L^2(\cG)}}{\|f\|^2_{L^2(\cG)}}.
\ee
Moreover, the bottom of the essential spectrum is given by 
\be\label{eq:Rayleighess}
\lambda_0^{\ess}(\bH):=\inf \sigma_{\ess}(\bH) = \sup_{\wt\cG\subset\cG} \inf_{\substack{f\in \wt{H}^1_0(\cG\setminus\wt\cG)\\ f\neq0}}\frac{\|f'\|^2_{L^2(\cG\setminus\wt\cG)}}{\|f\|^2_{L^2(\cG\setminus\wt\cG)}},
\ee
where the $\sup$ is taken over all finite subgraphs $\wt\cG$ of $\cG$. Here for any $\wt{\cG} \subset \cG$ we define $ \wt H^1_0(\cG\setminus\wt\cG)$ as the set of $H^1_0(\cG\setminus\wt\cG)$ functions satisfying the following boundary conditions: for vertices in  $\cG\setminus\wt\cG$ having one or more edges in $\wt \cG$, we change the boundary conditions from Kirchhoff--Neumann to Dirichlet; for all other vertices in $\cG \setminus \wt \cG$, we leave them the same. This equality is known as a Persson-type theorem (or Glazman's decomposition principle in the Russian literature, see \cite{gla}) and its proof in the case of quantum graphs is analogous to the case of Schr\"odinger operators (see, e.g., \cite[Theorem 3.12]{cfks}). 

\begin{remark}\label{rem:finitevol}
Let us mention that the following equivalence holds true
\be\label{eq:lam0=0}
\lambda_0(\bH) = 0 \qquad \Longleftrightarrow \qquad \lambda_0^{\ess}(\bH)=0.
\ee
The implication ``\ $\Leftarrow$\ " is obvious. However, $\lambda_0(\bH) = 0$ and $\lambda_0^{\ess}(\bH)\neq0$ holds only if $0$ is an isolated eigenvalue. On the other hand, \eqref{eq:gt00} implies that $0$ is an eigenvalue of $\bH$ only if $\id \in L^2(\cG)$. The latter happens exactly when 
\[
\mes(\cG):= \sum_{e\in\cE}|e|<\infty.
\]
and hence the equivalence \eqref{eq:lam0=0} holds true whenever $\mes(\cG)=\infty$, 
 
 On the other hand, it turns out that $\id\notin H^1_0(\cG)$ if $\mes(\cG)<\infty$ and hence $0$ is never an eigenvalue of $\bH$ (see Corollary \ref{cor:3.5}(iv)). In particular, the latter implies that $\gt_\cG\neq \gt_\cG^{(N)}$ if the metric graph $\cG$ has finite total volume, $\mes(\cG)<\infty$. The analysis of this case is postponed to a separate publication. 
\end{remark}

If $ \cG_1$, $ \cG_2$ are finite subgraphs with $ \cG_1 \subseteq  \cG_2 \subset \cG$, then $\wt{H}^1_0(\cG \setminus  \cG_2) \subseteq \wt{H}^1_0(\cG \setminus  \cG_1)$ in the sense that every function in $\wt{H}^1_0 (\cG \setminus  \cG_2)$ can be extended to be in $\wt{H}^1_0 (\cG \setminus  \cG_1)$ by setting it zero on remaining edges. Thus,
\[
		\inf_{\substack{f\in \wt{H}^1_0(\cG\setminus \cG_2)\\ f\neq0}}\frac{\|f'\|^2_{L^2(\cG \setminus\cG_2)}}{\|f\|^2_{L^2(\cG \setminus \cG_2)}} \geq
		\inf_{\substack{f\in \wt{H}^1_0(\cG\setminus \cG_1)\\ f\neq0}}\frac{\|f'\|^2_{L^2(\cG \setminus \cG_1)}}{\|f\|^2_{L^2(\cG\setminus\cG_1)}}.
\]
Let $\cK_\cG$ be the set of all finite, connected subgraphs of $\cG$ ordered by the inclusion relation ``$\subseteq$" and hence $\cK_\cG$ is a net. Moreover,
\begin{equation} \label{eq:lambdaesslimit}
	\lambda_0^{\ess}(\bH) = \sup_{\wt\cG \in \cK_\cG} \inf_{\substack{f\in \wt{H}^1_0(\cG\setminus\wt\cG)\\ f\neq0}}\frac{\|f'\|^2_{L^2(\cG\setminus\wt\cG)}}{\|f\|^2_{L^2(\cG\setminus\wt\cG)}} = \lim_{\wt \cG \in \cK_\cG} \inf_{\substack{f\in \wt{H}^1_0(\cG\setminus\wt\cG)\\ f\neq0}}\frac{\|f'\|^2_{L^2(\cG\setminus\wt\cG)}}{\|f\|^2_{L^2(\cG\setminus\wt\cG)}},
\end{equation}
where the limit is understood in the sense of nets and in this case we will say that $\wt\cG$ tends to $\cG$. 

The next result provides an estimate, which easily follows from \eqref{eq:Rayleigh}--\eqref{eq:Rayleighess}.

\begin{lemma}\label{lem:est01}
Set 
\be\label{eq:est02}
\ell^\ast_{\ess}(\cG) := \inf_{\wt\cE} \sup_{e\in\cE\setminus\wt\cE} |e|,
\ee
where the infimum is taken over all finite subsets $\wt\cE$ of $\cE$. Then
\begin{align}\label{eq:est01}
\lambda_0(\bH) & \le   
\frac{\pi^2}{\ell^\ast(\cG)^2}, & \lambda_0^{\ess}(\bH) & \le  
 \frac{\pi^2}{\ell^\ast_{\ess}(\cG)^2}.
\end{align}
\end{lemma}

\begin{proof}
By construction, the set $\wt{H}^1_c(\cG):=\wt{H}^1_0(\cG)\cap L^2_c(\cG)$ is a core for $\gt_\cG$. Moreover, every $f\in \wt{H}^1_0(\cG)$ admits a unique decomposition $f=f_{\rm lin} + f_0$, where $f_{\rm lin}\in \wt{H}^1_0(\cG)$ is piecewise linear on $\cG$ (that is, it is linear on every edge $e\in\cE$) and $f_0\in \wt{H}^1_0(\cG)$ takes zero values at the vertices $\cV$. It is straightforward to check that 
\be\label{eq:formdecomp}
\gt_{\cG}[f]=\int_{\cG} |f'(x)|^2 dx = \int_{\cG} |f_{\rm lin}'(x)|^2 dx + \int_{\cG} |f_0'(x)|^2 dx = \gt_{\cG}[f_{\rm lin}] + \gt_{\cG}[f_0]. 
\ee
Now the estimates \eqref{eq:est01} and \eqref{eq:est02} easily follow from the decomposition \eqref{eq:formdecomp}. Indeed, for every $f=f_0 \in \wt{H}^1_0(\cG)$
\be
\gt_\cG[f_0] = \sum_{e\in\cE} \|f_{0,e}'\|^2_{L^2(e)},
\ee
where $f_{0,e}:= f_0\upharpoonright e \in H^1_0(e)$. Noting that
\[
\inf_{f\in H^1_0([0,l])}\frac{\|f'\|^2_{L^2}}{\|f\|^2_{L^2}} = \left(\frac{\pi}{l}\right)^2,
\]
and then taking into account \eqref{eq:Rayleigh} and  \eqref{eq:Rayleighess}, we arrive at \eqref{eq:est01}. 
\end{proof}

\begin{remark}\label{rem:3.3}
A few remarks are in order:
\begin{itemize}
\item[(i)] The estimate \eqref{eq:est01} shows that the condition \eqref{eq:hyp1} is not a restriction since in the case $\ell^\ast(\cG)=\infty$ one immediately gets $\lambda_0(\bH)=\lambda_0^{\ess}(\bH)=0$. Moreover, in this case $\sigma(\bH)$ coincides with the positive semi-axis $\R_{\ge 0}$ (see \cite[Theorem 5.2]{sol03}).
\item[(ii)] The second inequality in \eqref{eq:est01} implies that \eqref{eq:est02} is necessary for the spectrum of $\bH$ to be purely discrete. Notice that $\ell^\ast_{\ess}(\cG)=0$ means that the number $\#\{e\in\cE|\, |e|>\varepsilon\}$ is finite for every $\varepsilon>0$.
\item[(iii)]
The estimates \eqref{eq:est01} can be slightly improved by noting that we can use other test functions on the edges $e\in\cE_N$  to improve the bound $(\pi/|e|)^2$ by $(\pi/2|e|)^2$. For example, we get the following estimate
\be\label{eq:est01B}
\lambda_0(\bH) \le  \min\Big\{ \inf_{e\in\cE\setminus\cE_N} \left(\frac{\pi}{|e|}\right)^2, \inf_{e\in \cE_N} \left(\frac{\pi}{2|e|}\right)^2\Big\}.
\ee
\end{itemize}
\end{remark}

%%%%%%%%%%%%%%%%%%%%%%%%%%%%%%%%%%%%%%%%%%%%%%%%%%%%%%%%%%%%
\subsection{Connection with the difference Laplacian}\label{ss:III.00}
%%%%%%%%%%%%%%%%%%%%%%%%%%%%%%%%%%%%%%%%%%%%%%%%%%%%%%%%%%%%

In this section we restrict for simplicity to the case of Neumann boundary conditions at the loose ends, that is, $f_e'(v)=0$ for all $v\in \cV$ with $\deg(v)=1$. 
Let the weight function $m\colon \cV\to \R_{>0}$ be given by \eqref{eq:m_def}. Consider the
difference Laplacian defined in $\ell^2(\cV;m)$ by the expression
\be\label{eq:tau}
(\tau_{\cG} f)(v) := \frac{1}{m(v)} \sum_{u\sim v} \frac{f(v) - f(u)}{|e_{u,v}|},\quad v\in\cV.
\ee
Namely, $\tau_\cG$ generates in $\ell^2(\cV;m)$ the pre-minimal operator 
\be\label{eq:h_def}
\begin{array}{cccc}
\rh_0  \colon & \dom(\rh_0) & \to & \ell^2(\cV;m) \\
 & f & \mapsto & \tau_\cG f 
\end{array},\qquad \dom(\rh_0):= C_c(\cV),
\ee
where $C_c(\cV)$ is the space of finitely supported functions on $\cV$. 
The operator $\rh_0$ is a nonnegative symmetric operator. Denote its Friedrichs extension by $\rh$. 

It was observed in \cite{ekmn} that the operators $\bH$ and $\rh$ are closely connected (for instance, by \cite[Corollary 4.1(i)]{ekmn}, $\bH_0$ and $\rh_0$ are essentially self-adjoint only simultaneously).  
In fact, it is not difficult to notice a connection between $\bH$ and $\rh$ by considering their quadratic forms (see \cite[Remark 3.7]{ekmn}). Namely, let $\cL= \ker(\bH_{\max})$ be the kernel of $\bH_{\max}$, which consists of piecewise linear functions on $\cG$. Every $f\in \cL$ can be identified with its values  $\{f(e_i), f(e_o)\}_{e\in \cE}$ on $\cV$ and, moreover,  
\be
\|f\|^2_{L^2(\cG)} = \sum_{e\in\cE} |e| \frac{|f(e_i)|^2 + \re(f(e_i)f(e_o)^\ast) + |f(e_o)|^2}{3}.
\ee
Now restrict ourselves to the subspace $\cL_{cont} = \cL\cap C_c(\cG)$. Clearly,  
\[
 \sum_{e\in\cE} |e| ({|f(e_i)|^2 + |f(e_o)|^2}) = \sum_{v\in\cV} |f(v)|^2\sum_{e\in \cE_v}|e|= \|f\|^2_{\ell^2(\cV;m)}
\]
defines an equivalent norm on $\cL_{cont}$ since the Cauchy--Schwarz inequality immediately implies
\be\label{eq:norms}
\frac{1}{6}\|f\|^2_{\ell^2(\cV;m)} \le \|f\|^2_{L^2(\cG)} \le \frac{1}{2}\|f\|^2_{\ell^2(\cV;m)}.
\ee
On the other hand, for every $f\in \cL_{cont}$ we get
\be\label{eq:spec_conn}
\begin{split}
\gt_\cG[f] = (\bH f,f)_{L^2(\cG)} & = \sum_{e\in\cE} \int_{e} |f'({x_e})|^2 d{x_e} = \sum_{e\in\cE} \frac{|f(e_o) - f(e_i)|^2}{|e| }\\
&=\frac{1}{2}\sum_{u,v\in \cV} \frac{|f(v) - f(u)|^2}{|e_{u,v}|} =(\rh f,f)_{\ell^2(\cV;m)}=:\gt_{\rh}[f].
\end{split}
\ee

Hence we end up with the following estimate.

\begin{lemma}\label{lem:estA}
\begin{align}\label{eq:estA}
 \lambda_0(\bH) & \le 6\lambda_0(\rh), & \lambda_0^{\ess}(\bH) & \le 6\lambda_0^{\ess}(\rh).
\end{align}
\end{lemma}

\begin{proof}
Clearly, the Rayleigh quotient \eqref{eq:Rayleigh} together with \eqref{eq:norms} and \eqref{eq:spec_conn} imply
\begin{align*}
\lambda_0(\bH) = \inf_{f\in H^1_0(\cG)} \frac{\gt_\cG[f]}{\|f\|^2_{L^2(\cG)}} &\le  \inf_{f\in \cL_{cont}} \frac{\gt_\cG[f]}{\|f\|^2_{L^2(\cG)}} \\
&\le \inf_{f \in C_c(\cV)} \frac{\gt_\rh[f]}{\frac{1}{6}\|f\|^2_{\ell^2(\cV;m)}} = 6\lambda_0(\rh).\qedhere
\end{align*}
\end{proof}

If $\cG$ is {\em equilateral} (that is, $|e|= 1$ for all $e\in \cE$), then $m(v) = \deg(v)$ for all $v\in\cV$ and hence $\tau_\cG$ coincides with the {\em combinatorial Laplacian}
\be\label{eq:tau_comb}
(\tau_{\rm comb} f)(v) := \frac{1}{\deg_\cG(v)} \sum_{u\sim v} f(v) - f(u),\quad v\in\cV.
\ee
In this particular case spectral relations between $\bH$ and $\rh$ have already been observed by many authors (see \cite{vB}, \cite[Theorem 1]{cat}, \cite{ex97} and \cite[Theorem 3.18]{bgp}). 

\begin{theorem}\label{th:below}
If $|e|=1$ for all $e\in\cE$, then 
\begin{align}\label{eq:conn_vBCat}
\lambda_0(\rh) & = 1 - \cos\big(\sqrt{\lambda_0(\bH)}\big), & \lambda_0^{\ess}(\rh) & = 1 - \cos\big(\sqrt{\lambda_0^{\ess}(\bH)}\big).
\end{align}
\end{theorem}

\begin{remark}
Actually, far more than \eqref{eq:conn_vBCat} is known in the case of equilateral quantum graphs. In fact, there is a sort of unitary equivalence between equilateral quantum graphs and the corresponding combinatorial Laplacians (see \cite{pan12,pan13} and also \cite{lp}).
\end{remark}

Hence for equilateral graphs we obtain 
\begin{align*}
\lambda_0(\rh) & \le \frac{1}{2}\lambda_0(\bH), & \lambda_0^{\ess}(\rh) & \le \frac{1}{2}\lambda_0^{\ess}(\bH). 
\end{align*}
The latter together with \eqref{eq:estA} imply that for equilateral graphs the following equivalence holds true
\begin{align}\label{eq:equivA}
 &\lambda_0(\bH) >0\quad \big(\lambda_0^{\ess}(\bH)>0\big)&\quad&\Longleftrightarrow&\quad&\lambda_0(\rh)>0\quad\big(\lambda_0^{\ess}(\rh)>0\big).&
\end{align}
In fact, it was proved recently in \cite[Corollary 4.1]{ekmn} that the equivalence \eqref{eq:equivA} holds true if the metric graph $\cG$ satisfies Hypothesis \ref{hyp:01}. Unfortunately, there is no such simple connection like \eqref{eq:conn_vBCat} if $\cG$ is not equilateral.

\begin{remark}\label{rem:connA}
Spectral gap estimates for combinatorial Laplacians is an established topic with a vast literature because of their numerous applications (see \cite{al,am,cdv,dsv,dod,fie,hoo,woe} and references therein). Recently there was a considerable interest in the study of spectral bounds for discrete (unbounded) Laplacians on weighted graphs (see \cite{bkw15,km}). On the one hand, \eqref{eq:conn_vBCat} and \eqref{eq:equivA} indicate that there must be analogous estimates for quantum graphs, however, we should stress that \eqref{eq:conn_vBCat} holds only for equilateral graphs. On the other hand, these connections also indicate that spectral estimates for quantum graphs should have a combinatorial nature.  
\end{remark}

\begin{remark}
Since $\frac{4}{\pi^2}x\le 1-\cos(\sqrt{x})$ for all $ x\in [0,{\pi^2}/{4}]$, 
 \eqref{eq:conn_vBCat} implies the following estimate for equilateral quantum graphs
\begin{align*}
 \lambda_0(\bH) & \le \frac{\pi^2}{4}\lambda_0(\rh), & \lambda_0^{\ess}(\bH) & \le \frac{\pi^2}{4}\lambda_0^{\ess}(\rh),
\end{align*}
which improves \eqref{eq:estA}.
Moreover, the constant $\pi^2/4$ is sharp in the equilateral case. However, it remains unclear to us how sharp is the estimate \eqref{eq:estA}. 
\end{remark}

%%%%%%%%%%%%%%%%%%%%%%%%%%%%%%%%%%%%%%%%%%%%%%%%%%%%%%%%%%%%
\section{The Cheeger-type bound}\label{ss:III.02}
%%%%%%%%%%%%%%%%%%%%%%%%%%%%%%%%%%%%%%%%%%%%%%%%%%%%%%%%%%%%

For every $\wt\cG\in\cK_\cG$ we define {\em the boundary of $\wt\cG$ with respect to the graph $\cG$} as the set of all vertices $v\in \wt\cV\setminus\cV_N$ such that either $\deg_{\wt\cG}(v)=1$ or $\deg_{\wt\cG}(v)<\deg_{\cG}(v)$, that is,
\be\label{eq:G_G}
\partial_\cG \wt{\cG} 
:= \big\{v\in \tilde{\cV} |\ v \in \cV_D \ \text{or} \ \deg_{\wt\cG}(v)<\deg_{\cG}(v)\big\}.
\ee
For a given finite subgraph $\wt\cG\subset \cG$ we then set 
\be\label{eq:deg_wtG}
\deg(\partial_\cG \wt{\cG}) := \sum_{v\in\partial_\cG \wt{\cG}} \deg_{\wt\cG}(v).
\ee

\begin{remark}
Let us stress that our definition of a boundary is different from the combinatorial one. In particular, we define the boundary as the set of vertices whereas the combinatorial definition counts the number of edges connecting $\wt\cV$ with its complement $\cV\setminus\wt\cV$.
\end{remark}

\begin{definition}\label{def:iso}
\begin{itemize}
\item[(i)] {\em The isoperimetric (or Cheeger) constant} of a metric graph $\cG$ is defined by
\be\label{eq:alphaG}
		\alpha(\cG) := \inf_{\wt{\cG} \in \cK_\cG} \; \frac{ \deg(\partial_\cG \wt{\cG}) }{\mes(\wt{\cG})} \in [0, \infty),
\ee
where $\mes(\wt\cG)$ denotes the Lebesgue measure of $\wt\cG$, 
$ \mes(\wt\cG) := \sum_{e\in\wt\cE} |e|$.
\item[(ii)] {\em The isoperimetric constant at infinity} is defined by
\be\label{eq:alphaessG}
		\alpha_{\ess}(\cG) := \sup_{\wt\cG \in \cK_\cG} \alpha(\cG\setminus \wt\cG) \in [0,\infty].
\ee
\end{itemize}
\end{definition}

Recall that for any $\wt\cG \in \cK_\cG$ we consider $ \cG\setminus \wt\cG$ with the following boundary conditions: for vertices in  $\cG\setminus \wt\cG$ having one or more edges in $ \wt\cG$, we change the boundary conditions from Kirchhoff--Neumann to Dirichlet; for all other vertices in $\cG \setminus  \wt\cG$, we leave them the same. These boundary conditions imply that for a subgraph $\cY \in \cK_{\cG \setminus \wt\cG}$,
\be\label{eq:boundary_subG}
\partial_{\cG \setminus \wt\cG} {\cY} = \partial_{\cG} {\cY},
\ee
where the left-hand side is the boundary of $\cY$ with respect to $\cG \setminus  \wt\cG$ (with the new Dirichlet conditions) and the right-hand side is the boundary with respect to the original graph $\cG$. Hence,
\[
		\alpha(\cG \setminus  \wt\cG ) = \inf_{\cY \in \cK_{ \cG \setminus  \wt\cG} }\; \frac{ \deg (\partial_{ \cG \setminus  \wt\cG} \cY  ) }{\mes( \cY )} = \inf_{\cY \in \cK_{ \cG \setminus  \wt\cG} }\; \frac{ \deg (\partial_{ \cG } \cY  ) }{\mes( \cY )}
\]
and $\alpha(\cG \setminus \cG_1 ) \leq \alpha(\cG \setminus  \cG_2 )$ whenever $ \cG_1 \subseteq  \cG_2$. Thus,
\begin{equation}\label{eq:alphaesslimit}
  \alpha_{\ess}(\cG) =  \sup_{ \wt\cG \in \cK_\cG} \alpha(\cG\setminus  \wt\cG) = \lim_{ \wt\cG \in \cK_\cG}  \alpha(\cG \setminus  \wt\cG).
\end{equation}

\begin{remark}
Choosing $\wt\cG$ as an edge $e\in \cE$ or a star $\cE_v$ with some $v\in\cV$, one gets the following simple bounds on the isoperimetric constant
\begin{align}\label{eq:alphaGest}
\alpha(\cG) & \le \frac{2}{\ell^\ast(\cG)}, & \qquad \alpha(\cG) & \le \inf_{v\in\cV} \frac{\deg_\cG(v)}{m(v)}.
\end{align}
\end{remark}

The next result is the analog of the famous Cheeger estimate for Laplacians on manifolds \cite{che}.

\begin{theorem}\label{th:Cheeger}
\begin{align}\label{eq:Cheeger}
 \lambda_0(\bH) & \ge \frac{1}{4}\alpha(\cG)^2, & \lambda_0^{\ess}(\bH) & \ge \frac{1}{4}\alpha_{\ess}(\cG)^2. 
\end{align}
\end{theorem}

As an immediate corollary we get the following result.

\begin{corollary}\label{cor:3.5}
\begin{itemize}
\item[(i)] $\bH$ is uniformly positive whenever $\alpha(\cG)>0$.
\item[(ii)]  $\lambda_0^{\ess}(\bH)>0$  if $\alpha_{\ess}(\cG)>0$. 
\item[(iii)] The spectrum of $\bH$ is purely discrete if  $\alpha_{\ess}(\cG)=\infty$.
\item[(iv)] If the metric graph $\cG$ has finite total volume, $\mes(\cG)<\infty$, then $\bH$ is a uniformly positive operator with purely discrete spectrum.
\end{itemize}
\end{corollary}

\begin{proof}
Clearly, we only need to prove (iv). Since $\mes(\cG)<\infty$ and taking \eqref{eq:alphaG} into account, we immediately obtain 
\begin{align}\label{eq:mGfinite}
\alpha(\cG)  \ge \frac{1}{\mes(\cG)}, 
\end{align}
which together with \eqref{eq:Cheeger} implies the inequality $\lambda_0(\bH)>0$. Next, using \eqref{eq:alphaessG} together with the estimate \eqref{eq:mGfinite} and the net property of $\cK_\cG$, one gets $\alpha_{\ess}(\cG) =\infty$, which finishes the proof.
\end{proof}

Before proving the estimates \eqref{eq:Cheeger} we need several preliminary lemmas. In what follows, for every $U\subseteq \cG$, we shall denote by $\partial U$ the boundary of a set $U$ in the sense of the natural metric topology on $\cG$ (see Section \ref{ss:II.01}). 
For every measurable function $h\colon\cG\to \R$ and every $t\in\R$ let us define the set 
\be\label{eq:Omega_h}
	\Omega_h(t):= \{ x \in \cG| \; h(x)>t\}.
\ee	
The next statement is known as the {\em co-area formula} and we give its proof for the sake of completeness.

\begin{lemma}\label{lem:coareaqg}
		If $h \colon \cG \to \R$ is continuous on $\cG$ and continuously differentiable on every edge $e\in\cE$, then
		\begin{equation}\label{eq:coarea}
			\int_{\cG} |h'(x)| \; dx = \int_\R \#( \partial  \Omega_h(t) ) \; dt.
		\end{equation}
\end{lemma}
		
\begin{proof}
		Assume first that ${\rm supp}(h)\subset e$ for some $e\in \cE$. We can identify $e$ with the open interval $(0, |e|)$ and hence 
\[
	M_e := \{x \in e| \; h'(x) \neq0\}
\] 
can be written as $M_e = \bigcup_{n} I_n$ for (at most countably many) disjoint open intervals $I_n \subseteq (0, |e|)$. Since $h$ is strictly monotone on each of these intervals,
		\begin{align*}
		\int_\cG |h'(x)| \; dx&=	\int_e |h'(x)| \; dx =\int_{M_e} |h'(x)| \; dx \\
		&= \sum_{n} \int_{I_n} |h'(x)|\; dx =  \sum_{n} \mes(h(I_n)) =  \sum_{n} \int_\R \id_{h(I_n)}(s) \; ds.
		\end{align*}
Here $\mes(X)$ denotes the Lebesgue measure of $X\subseteq \R$. Moreover, by continuity of $h$, it is straightforward to check that $ \id_{h(I_n)}(t)  =  \#( \partial \Omega_h(t) \cap I_n)$ for all $t\in \R$.  
Hence we end up with
		\begin{equation*}
\sum_{n} \int_\R \id_{h(I_n)}(t) \; dt = \sum_{n} \int_\R \#( \partial \Omega_h(t) \cap I_n) \; dt 
		= \int_\R \#( \partial \Omega_h(t) \cap M_e) \; dt.
		\end{equation*}
		
		Now assume that $t \in \R$ is such that $\partial \Omega_h(t) \cap M_e^c \neq \varnothing$, where 
\[
M_e^c:=e\setminus M_e = \{x \in e| \; h'(x) = 0\}
\] 
is the set of critical points of $h$. 
 By Sard's Theorem \cite{Sard}, $h(M_e^c)$ has Lebesgue measure zero and hence
 \[
 \int_\R \#( \partial \Omega_h(t) \cap M_e) \; dt = \int_\R \#( \partial \Omega_h(t) \cap e) \; dt.
 \]

Assume now that $h\colon \cG\to \R$ is an arbitrary function satisfying the assumptions. Then we get
		\begin{align*}
		\int_{\cG}  |h'(x)| \; dx &= \sum_{e \in \cE} \int_e |h'(x)| \; dx \\
		&=  \sum_{e \in \cE} \int_\R \#( \partial \Omega_h(t) \cap e) \; dt 
		= \int_\R \#( \partial \Omega_h(t) \cap (\cG \backslash \cV)) \; dt.
		\end{align*}
		If $\partial \Omega_h(t) \cap \cV\neq \varnothing$, then $t \in h(\cV)$. Since $\cV$ is countable, we arrive at \eqref{eq:coarea}.
\end{proof}

Next it will turn out useful to rewrite the Cheeger constant \eqref{eq:alphaG} in the following way. Let 
	\begin{equation}\label{eq:alphaG02}
		\wt{\alpha}(\cG) := \inf_{U \in \cU_\cG} \frac{\#(\partial U)}{\mes (U)},
	\end{equation}
where $\cU_\cG = \cup_{\wt\cG\in \cK_\cG} \cU_{\wt \cG}$ and
	\begin{equation}\label{eq:setU}
		\cU_{\wt{\cG}} = \{ U \subseteq \wt{\cG} | \; U \text{ is open}, \;   U \cap \cV_D = \varnothing \text{ and }  \partial U \cap \cV= \varnothing \}.
	\end{equation}

\begin{lemma} \label{lem:chopensets}
Let $\alpha(\cG)$ be defined by \eqref{eq:alphaG}. Then 
\be\label{eq:A=A}
\alpha(\cG) = \wt\alpha(\cG).
\ee
\end{lemma}

\begin{proof}
(i) It easily follows from the definition of $\wt\alpha(\cG)$ that 
\[
\wt\alpha(\cG) \le \alpha(\cG).
\]
Indeed, assume first that $\wt \cG \in \cK_\cG$ and identify $\wt \cG$ with a closed subset of the graph. For a sufficiently small 
	$\varepsilon >0$, we cut out a ball $B_\varepsilon(v)$ of radius $\varepsilon$ at each point in $v\in \partial_\cG \wt \cG$ and obtain the set
	\[ 
U := \wt \cG \backslash \bigcup_{v \in \partial_\cG \wt G } B_\varepsilon(v).
        \]
We have $U \in \cU_\cG$ and, moreover, $\partial U$ has precisely $\deg(\partial_\cG \wt{\cG})$ points. 
	In total,
	\[ 
\frac{\#(\partial U)}{\mes(U)} = \frac{\deg(\partial_\cG \wt \cG)}{\mes(\wt \cG ) - \varepsilon \deg(\partial_\cG \wt \cG)}. 
        \]
Letting $\varepsilon$ tend to zero, we obtain the desired inequality.

(ii) To prove the other inequality,  let $U \in \cU_\cG$ and $\wt \cG = (\wt \cV, \wt \cE)$ be the finite subgraph consisting of all edges $e\in\cE$ with $e \cap U \neq \varnothing$ and all vertices incident to such an edge. Clearly, $ \mes(U) \le \mes(\wt \cG)$. Also, by \eqref{eq:deg_wtG}, 
	\begin{align*}
			\deg(\partial_\cG \wt{\cG}) = \sum_{v \in \partial \wt \cG} \deg_{\wt \cG} (v) 
			= &\#\{ e \in \wt \cE| \; e \text{ connects } \partial_\cG \wt \cG \text{ and } \wt \cG\backslash \partial_\cG \wt\cG \} \\
			&+  2 \#\{ e \in \wt \cE| \; \text{ both vertices are in } \partial_\cG \wt \cG \}.
	\end{align*}
	Since $U$ is open,  every point of $\partial_\cG \wt \cG$ is not in $U$. Therefore, every edge in the subgraph $\wt\cG$ connected to a vertex in $\partial_\cG \wt \cG$ must contain at least one boundary point of $U$. If both vertices of the edge are in $\partial_\cG \wt \cG$, it must even contain at least two boundary points of $U$. Also, since $\cV \cap \partial U = \varnothing$, the boundary points lie in the strict interior of each edge and therefore cannot coincide for different edges.  Thus, $\deg(\partial_\cG \wt\cG) \leq \#(\partial U)$. 
	
	Finally, notice that $\wt\cG$ might be disconnected. If it is the case, then write $\wt \cG = \dot{\cup}_{n} \wt \cG_n$ as a disjoint, finite union of connected subgraphs $\wt \cG_n\in\cK_\cG $. Then
	\begin{equation*}
	\frac{\#(\partial U)} {\mes(U)}
	\geq 	\frac{\deg(\partial_\cG \wt \cG)}{\mes(\wt \cG)}
	= \frac{\sum_n \deg(\partial_\cG \wt \cG_n)} {\sum_n \mes(\wt \cG_n)}
	\geq \min_{n} \frac{ \deg(\partial_\cG \wt \cG_n)}{\mes(\wt \cG_n)},
	\end{equation*}
	which implies that $\wt\alpha(\cG) \ge \alpha(\cG)$.
\end{proof}

Now we are in position to prove the Cheeger-type estimates \eqref{eq:Cheeger}.

\begin{proof}[Proof of Theorem \ref{th:Cheeger}]
	Let us show that the following inequality
\be\label{eq:Rayleigh2}
\alpha(\cG)\,\|f\|_{L^2(\cG)} \le 2\|f'\|_{L^2(\cG)}
\ee
holds true for all $f\in \dom(\gt_\cG^0)=\dom(\bH_0)$. Without loss of generality we can restrict ourselves to real-valued functions. 
So, suppose $f\in \dom(\bH_0)$ is real-valued. Observe that (see, e.g., \cite[Lemma I.4.1]{Garnett})
\[
\|f\|^2_{L^2(\cG)} = \int_{\cG} f(x)^2 \; dx = \int_0^\infty \mes( \Omega_{f^2}(t)) \; dt.
\]
Next we want to use Lemma \ref{lem:chopensets} with $h=f^2$. If $t >0$ is such that $\partial \Omega_{f^2}(t) \cap \cV \neq \varnothing$, then $t \in f^2(\cV)$ by continuity of $f^2$. Since $\cV$ and hence $f^2(\cV)$ are countable, we get that $\Omega_{f^2}(t) \in \cU_\cG$ for almost every $t >0$. Thus, in view of Lemma \ref{lem:chopensets}
	\be
			\alpha(\cG) \|f\|_{L^2}^2 \leq  \int_0^\infty \#(\partial \Omega_{f^2}(t) ) \; dt.
	\ee
On the other hand, applying Lemma \ref{lem:coareaqg} to $h=f^2$ and then the Cauchy--Schwarz inequality, we get
\be
\int_0^\infty \#(\partial \Omega_{f^2}(t)) dt = 2\int_{\cG}|f(x)f'(x)| dx \le 2\|f\|_{L^2(\cG)}\|f'\|_{L^2(\cG)}.
\ee
Combining the last two inequalities, we arrive at \eqref{eq:Rayleigh2}, which together with the Rayleigh quotient \eqref{eq:Rayleigh} proves the first inequality in \eqref{eq:Cheeger}. 

The proof of the second inequality in \eqref{eq:Cheeger} follows the same line of reasoning since by \eqref{eq:Rayleighess}
	\[ 
\lambda_0^{\ess}(\bH)\geq  \inf_{\substack{f\in \wt{H}^1_0(\cG\setminus\wt\cG)\\ f\neq0}}\frac{\|f'\|^2_{L^2(\cG\setminus\wt\cG)}}{\|f\|^2_{L^2(\cG\setminus\wt\cG)}},
        \]
for every finite subgraph $\wt\cG$ of $\cG$. Notice that the boundary conditions on $\cG\setminus\wt{\cG}$ are defined after \eqref{eq:alphaessG}.
\end{proof}

\begin{remark}
 The Cheeger estimate for finite quantum graphs was first proved in \cite{nic} (see also \cite[\S 6]{post09} and \cite{km}). Our result extends \cite[Theorem 3.2]{nic} to the case of infinite graphs and also provides a bound on the essential spectrum of $\bH$. However, our definition of the isoperimetric constant \eqref{eq:Cheeger} is purely combinatorial since the infimum is taken over finite connected subgraphs of $\cG$, although the definition in \cite{nic} (see also \cite{km16,post09}) is similar to \eqref{eq:alphaG02}.

Let us mention that one can obtain a similar statement for the operator $\bH^N$ that is related to the maximally defined quadratic form (see Remark \ref{rem:maxform}). 
However, one needs to take the infimum in the definition of the isoperimetric constant over all subgraphs of finite volume.
\end{remark}

Taking into account the equivalence \eqref{eq:lam0=0}, let us finish this section with the next observation. 

\begin{lemma}\label{lem:equivalence}
The following equivalence holds true
\begin{align}\label{eq:equivalence}
\alpha(\cG) =0  \qquad  \Longleftrightarrow  \qquad \alpha_{\ess}(\cG)  =0.
\end{align}
\end{lemma}

\begin{proof}
Clearly, we only need to prove the implication $\alpha(\cG) =0\ \Rightarrow\ \alpha_{\ess}(\cG) =0$.
Assume the converse, that is, there is an infinite graph $\cG$ satisfying Hypotheses \ref{hyp:locfin}--\ref{hyp:02} such that $\alpha(\cG) =0$ and $\alpha_{\ess}(\cG)>0$.  
 Then by \eqref{eq:alphaG}, there is a sequence $\{\cG_n\}\subset \cK_\cG$ such that
\[
\alpha(\cG) = \lim_{n\to \infty} \frac{\deg(\partial_\cG \cG_n)}{\mes(\cG_n)} = 0.
\]
On the other hand, \eqref{eq:alphaessG} implies that there is $\wt\cG\in\cK_\cG$ such that $\alpha(\cG\setminus\wt\cG) = \alpha_0>0$.
In particular, taking into account \eqref{eq:boundary_subG}, the latter is equivalent to the fact that
\[
 \frac{\deg(\partial_{\cG\setminus\wt\cG} \cY)}{\mes(\cY)} =  \frac{\deg(\partial_{\cG} \cY)}{\mes(\cY)} \ge \alpha_0>0
\]
for every finite subgraph $\cY\subset \cG\setminus\wt\cG$.

Next observe that 
\[
\lim_{n\to \infty} \frac{\deg(\partial_\cG (\cG_n\setminus\wt\cG))}{\mes(\cG_n\setminus\wt\cG)} = 0,
\]
which leads to a contradiction. 
Indeed, by construction, $\lim_{n\to \infty}\mes(\cG_n) = \infty$ and hence $\mes(\cG_n\setminus\wt\cG) = \mes(\cG_n)(1+o(1))$  as $n\to \infty$. 
It remains to note that
\[
\deg(\partial_\cG\cG_n) - \deg(\wt\cG) \le \deg(\partial_\cG (\cG_n\setminus\wt\cG)) \le \deg(\partial_\cG\cG_n) + \deg(\wt\cG).\qedhere
\]
\end{proof}

%%%%%%%%%%%%%%%%%%%%%%%%%%%%%%%%%%%%%%%%%%%%%%%%%%%%%%%%%%%%
\section{Connections with discrete isoperimetric constants}\label{ss:III.03}
%%%%%%%%%%%%%%%%%%%%%%%%%%%%%%%%%%%%%%%%%%%%%%%%%%%%%%%%%%%%

For every vertex set $X \subseteq \cV$, we define its boundary and interior edges by
\begin{align*}
		\cE_b(X) &= \{e \in \cE| \; e \text{ connects } X \text{ and } \cV \backslash X \},  \\
		\cE_i(X) &= \{ e \in \cE| \; \text{ all vertices incident to } e \text{ are in } X \}.
\end{align*}
Also, for a vertex set $X \subseteq \cV$ we set
\begin{equation*}
		m(X) := \sum_{v \in X} m(v),
\end{equation*}
where $m\colon\cV\to (0,\infty)$ is defined by \eqref{eq:m_def} (in fact, $m(v) = \mes(\cE_v)$ for every $v\in\cV$). {\em The (discrete) isoperimetric constant} $\alpha_d(Y)$ of $Y\subseteq \cV$ is defined by
\be\label{eq:discrCheegerY}
			\alpha_d(Y) := \inf_{ \substack{X \subseteq Y \\X \text{\ is finite}}} \; \frac{\#(\cE_b(X))}{m(X)} \in [0, \infty).
\ee
{\em The discrete isoperimetric constant} of the graph $\cG$ is then given by
\be\label{eq:discrCheeger}
			\alpha_d(\cV) := \inf_{ \substack{X \subseteq \cV \\X \text{\ is finite}}} \; \frac{\#(\cE_b(X))}{m(X)} \in [0, \infty).
\ee
Moreover, we need the {\em discrete isoperimetric constant at infinity}
\be\label{eq:discrCheegeress}
			\alpha_d^{\ess}(\cV) := \sup_{ \substack{X \subseteq \cV \\X \text{\ is finite}}} \alpha_d(\cV\setminus X) \in [0, \infty].
\ee

\begin{remark}
Our definition of the isoperimetric constants follows the one provided in Appendix~\ref{sec:app} (see Remark \ref{rem:ourLaplace}). This definition is slightly different from the one given in \cite{bkw15}, which uses the notion of an intrinsic metric on $\cV$ (cf. \cite{flw14}). In particular, the natural path metric $\varrho_0$ (cf. Section~\ref{ss:II.03}) is intrinsic in the sense of \cite{bkw15, flw14} and in certain cases (if, for example, $\cG_d$ is a tree) the corresponding definitions from \cite{bkw15} coincide with \eqref{eq:discrCheeger} and \eqref{eq:discrCheegeress}. Notice that the following Cheeger-type estimates for the discrete Laplacian \eqref{eq:tau}--\eqref{eq:h_def} (see \cite[Theorems 3.1 and 3.3]{bkw15} and Theorem \ref{th:cheegerdiscr}) hold true
\begin{align}\label{eq:Cheegerdiscr}
\lambda_0(\rh) & \ge \frac{1}{2}\alpha_d(\cV)^2, & \lambda_0^{\ess}(\rh) & \ge \frac{1}{2}\alpha^{\ess}_d(\cV)^2.
\end{align}
\end{remark}

The next result provides a connection between isoperimetric constants. 

\begin{lemma}\label{lem:Cheegerconnection}
	The isoperimetric constants \eqref{eq:alphaG} and \eqref{eq:discrCheeger} can be related by 
	\begin{align}\label{eq:connCheeger}
\frac{1}{2}\alpha(\cG) & \le \alpha_d(\cV), & \frac{2}{\alpha(\cG)} & \leq  \frac{1}{\alpha_d(\cV)} + \ell^\ast(\cG).
	 \end{align}
In particular, the isoperimetric constants at infinity \eqref{eq:alphaessG} and \eqref{eq:discrCheegeress} satisfy
	 \begin{align}\label{eq:connCheegerEss} 
\frac{1}{2}\alpha_{\ess}(\cG) & \le \alpha^{\ess}_d(\cV), & \frac{2}{\alpha_{\ess}(\cG)} & \leq  \frac{1}{\alpha_d^{\ess}(\cV)} + \ell^\ast_{\ess}(\cG).
	 \end{align}
\end{lemma}

\begin{proof}
(i)		First, let $X \subset \cV$ be finite. Let also $\wt \cG = ( \wt \cV, \wt \cE)$ be the finite subgraph of $\cG$ consisting of all edges with at least one vertex in the set $X$. Observe that 
		\[
		\wt \cE = \bigcup_{v\in X} \cE_v =  \cE_i (X) \cup \cE_b(X).
		\] 
Then
		\[ 
		m(X) = \sum_{v \in X} m(v) = 2 \sum_{e \in \cE_i(X)} |e| + \sum_{e \in \cE_b(X)} |e| \leq 2 \sum_{e \in \wt \cE} |e| = 2\, \mes(\wt \cG).
		\]
Note that for every $v \in X$, the whole star $\cE_v$ attached to it is in $\wt \cG$. Therefore, every vertex from $\partial_\cG\wt \cG$ is not in $X$. Now consider an edge $e$ in the subgraph $\wt \cG$ which is connected to a vertex $v\in \partial_\cG\wt \cG$. Then its other endpoint must be in $X$ (because of the definition of $\wt \cG$). Hence
		\begin{align*}
		\deg(\partial_\cG \wt{\cG}) &= \sum_{v\in\partial \wt{\cG}} \deg_{\wt\cG}(v) = \sum_{v\in\partial \wt{\cG}} \#\{e | \; e \text{ connects } v \text{ and } X \} \\
		&\leq \#\{e \in\wt\cE| \; e \text{ connects }  X \text{ and } \cV \backslash X \} = \#( \cE_b(X)).
		\end{align*}
		Splitting $\wt \cG$ in finitely many connected components as in the proof of Lemma \ref{lem:chopensets}, we arrive at the first inequality in \eqref{eq:connCheeger}.
		
		To prove the second inequality, assume $\wt \cG \in \cK_\cG$. Write $\wt \cE = \wt \cE_0 \cup \wt \cE_1 \cup \wt \cE_2$, where $\wt \cE_0$, $\wt \cE_1$, $\wt \cE_2$ are the sets of edges in the subgraph with, respectively, none, one, and two vertices in $\partial_\cG \wt{\cG}$. Clearly,
		\begin{align}\label{eq:IV.06}
					\deg(\partial_\cG \wt{\cG}) = \#(\wt \cE_1) + 2 \#(\wt \cE_2).
		\end{align}
		Now define the finite vertex set $X := \wt \cV \backslash \partial_\cG \wt{\cG}$. We have
		\begin{align*}
				\cE_i (X) & = \wt \cE_0, & \cE_b (X) & = \wt \cE_1.
		\end{align*}
		Thus,
		\begin{align*}
				2 \frac{\mes(\wt \cG) }{\deg(\partial_\cG \wt{\cG})} &= 2 \frac{\sum_{e\in\wt \cE_0} |e| + \sum_{e\in\wt \cE_1} |e| + \sum_{e\in\wt \cE_2} |e|} {\#(\wt \cE_1) + 2 \#(\wt \cE_2)} \\
				&= \frac{  2  \sum_{e\in\cE_i(X)} |e| + \sum_{e\in\cE_b(X)} |e| } {\#(\cE_b(X)) + 2 \#(\wt \cE_2)} + \frac{\sum_{e\in\cE_b(X)} |e| + 2 \sum_{e\in\wt \cE_2} |e| }{\#(\cE_b(X)) + 2 \#(\wt \cE_2)} \\
				&= \frac{ m(X) } {\#(\cE_b(X)) + 2 \#(\wt \cE_2)} + \frac{\sum_{e\in\cE_b(X)} |e| + 2 \sum_{e\in\wt \cE_2} |e| }{\#(\cE_b(X)) + 2 \#(\wt \cE_2)}\\
				&\leq \frac{ m(X) } {\#(\cE_b(X))} + \frac{\sum_{e\in\cE_b(X)} |e| + 2 \sum_{e\in\wt \cE_2} |e| }{\#(\cE_b(X)) + 2 \#(\wt \cE_2)} \leq  \frac{ m(X) } {\#(\cE_b(X))} + \sup_{e\in \cE} |e|.
		\end{align*}

(ii) To prove \eqref{eq:connCheegerEss}, let first $X \subseteq \cV$ be a finite and connected (in the sense that for two vertices in $X$, there always exists a path connecting them and only passing through vertices in $X$) set of vertices. Then the subgraph $\wt\cG_X \subseteq \cG$ consisting of all edges with both vertices in $X$ is finite and connected. 
Now note that for a finite vertex set $Y \subseteq \cV \setminus X$, the subgraph $\wt \cG_Y$ defined above is contained in $\cG \setminus \wt\cG_X$. Hence taking into account \eqref{eq:boundary_subG} and using the same line of reasoning as in (i), we get $\alpha( \cG \setminus \wt\cG_X) \le 2\alpha_d (\cV \setminus X)$.  
Finally, choose an increasing sequence $\{X_n\} \subseteq \cV$ of finite and connected vertex sets such that every finite vertex set $X \subseteq \cV$ is eventually contained in $X_n$. Then the corresponding sequence $\{\wt\cG_n\} \subseteq \cK_{\cG}$ of subgraphs is increasing and every finite, connected subgraph $\wt\cG\in\cK_\cG$ is eventually contained in $\wt\cG_n$. In view of \eqref{eq:alphaesslimit}, we obtain the first inequality in \eqref{eq:connCheegerEss} by taking limits. 

To prove the second, for a subgraph $\cG_0 \in  \cK_\cG$, choose $X$ to be the set of vertices in $\cG_0$. Let $\wt \cG \in \cK_{\cG \setminus \cG_0}$. If a vertex $v$ is both in $\wt \cV$ and in $X$, then it has at least one incident edge which lies in the cut out graph $\cG_0$ and therefore $v \in \partial_\cG \wt \cG$. Thus, the vertex set $Y = \wt \cV \backslash \partial_\cG \wt{\cG}$ satisfies $Y \cap X = \varnothing$. Refining the previous estimate,
\[
	2 \frac{\mes(\wt \cG) }{\deg(\partial_\cG \wt{\cG})} \leq \frac{ m(Y) } {\#(\cE_b(Y))} + \frac{\sum_{e\in\cE_b(Y)} |e| + 2 \sum_{e\in\wt \cE_2} |e| }{\#(\cE_b(Y)) + 2 \#(\wt \cE_2)} \leq  \frac{ m(Y) } {\#(\cE_b(Y))} + \ell^\ast(\cG \setminus \cG_0),
\]
and hence
\[
 \frac{2}{\alpha (\cG \setminus \cG_0)} \leq \frac{1}{\alpha_d(\cV \setminus X)} + \ell^\ast(\cG \setminus \cG_0).
\]
Choosing an increasing sequence $\{\cG_n\} \subseteq \cK_\cG$ such that every $\cG_0 \in \cK_\cG$ is eventually contained in $\cG_n$ and applying the same limit argument as before, we arrive at the second inequality in \eqref{eq:connCheegerEss}.
\end{proof}

\begin{remark}
It can be seen by examples that the estimates \eqref{eq:connCheeger} and \eqref{eq:connCheegerEss} are sharp. Indeed, on the equilateral Bethe lattice (see Example \ref{ex:tree01}), one gets equalities in the second inequalities \eqref{eq:connCheeger} and \eqref{eq:connCheegerEss} (cf. \eqref{eq:alphaBethe}). 
\end{remark}

Combining \eqref{eq:connCheeger} with Corollary \ref{cor:3.5}, we obtain Theorem 4.18 from \cite{ekmn}.

\begin{corollary}[\cite{ekmn}]\label{th:H0positive}
\begin{itemize}
\item[(i)] $\lambda_0(\bH)>0$ if $\alpha_d(\cV)>0$.
\item[(ii)]  $\lambda_0^{\ess}(\bH)>0$  if $\alpha_d^{\ess}(\cV)>0$. 
\item[(iii)] The spectrum of $\bH$ is purely discrete if the number $\#\{e\in\cE\colon |e|>\varepsilon\}$ is finite for every $\varepsilon>0$ and $\alpha_d^{\ess}(\cV)=\infty$.
\end{itemize}
\end{corollary}

\begin{proof}
We only need to mention that $\ell^\ast_{\ess}(\cG) = 0$ if and only if the number $\#\{e\in\cE\colon |e|>\varepsilon\}$ is finite for every $\varepsilon>0$. Moreover, in this case it follows from \eqref{eq:connCheegerEss} that $\alpha^{\ess}(\cG) = \alpha_d^{\ess}(\cV)$.
\end{proof}

Finally, let us mention that in the case of equilateral graphs the discrete isoperimetric constants coincide with the combinatorial isoperimetric constants introduced in \cite{dk86}:
\begin{align}\label{eq:alpha_comb}
\alpha_{\comb}(\cV) & = \inf_{ \substack{X \subseteq \cV \\X \text{\ is finite}}} \frac{\#(\partial X)}{\deg(X)}, & \alpha_{\comb}^{\ess}(\cV) & = \sup_{ \substack{X \subseteq \cV \\X \text{\ is finite}}} \alpha_{\comb}(\cV\setminus X)
\end{align}
Comparing \eqref{eq:alpha_comb} with \eqref{eq:discrCheeger} and \eqref{eq:discrCheegeress} and noting that 
\[
\ell_\ast(\cG)\deg_\cG(v)\le m(v) \le \ell^\ast(\cG)\deg_\cG(v)
\]
for all $v\in\cV$, one easily derives the estimates 
\begin{align*} 
\frac{\alpha_{\comb}(\cV)}{\ell^\ast(\cG)} & \le \alpha_d(\cV) \le \frac{\alpha_{\comb}(\cV)}{\ell_\ast(\cG)}, & \frac{\alpha_{\comb}^{\ess}(\cV)}{\ell_{\ess}^\ast(\cG)} & \le \alpha_d^{\ess}(\cV) \le \frac{\alpha_{\comb}^{\ess}(\cV)}{\ell^{\ess}_\ast(\cG)}.
\end{align*}
Here
\be\label{eq:ell_*ess}
\ell_\ast^{\ess}(\cG) := \sup_{\wt\cE} \inf_{e\in\cE\setminus\wt\cE} |e|,
\ee
and the supremum is taken over all finite subsets $\wt\cE$ of $\cE$. 
Moreover, taking into account Lemma \ref{lem:Cheegerconnection}, we get the following connection between our isoperimetric constants and the combinatorial ones: 
\be\label{eq:A=Acomb01}
\frac{2\,\alpha_{\comb}(\cV)}{\ell^\ast(\cG)(1+ \alpha_{\comb}(\cV))} \le \alpha(\cG) \le \frac{2\,\alpha_{\comb}(\cV)}{\ell_\ast(\cG)}
\ee
and
\be\label{eq:A=Acomb02} 
\frac{2\,\alpha^{\ess}_{\comb}(\cV)}{\ell_{\ess}^\ast(\cG)(1+ \alpha^{\ess}_{\comb}(\cV))} \le \alpha^{\ess}(\cG) \le \frac{2\,\alpha_{\comb}^{\ess}(\cV)}{\ell^{\ess}_\ast(\cG)}.
\ee
Since $\alpha_{\comb}(\cV)\in [0,1)$, we end up with the following result.

\begin{corollary}\label{cor:alpha_comb}
Let $\cG$ be a metric graph such that $\ell^\ast(\cG)<\infty$. Then:
\begin{itemize}
\item[(i)] $\lambda_0(\bH)>0$ if $\alpha_{\comb}(\cV)>0$. 
\item[(ii)] $\lambda_0^{\ess}(\bH)>0$ whenever $\alpha_{\comb}^{\ess}(\cV)>0$.
\item[(iii)] The spectrum of $\bH$ is purely discrete if $\ell^\ast_{\ess}(\cG) = 0$ and $\alpha_{\comb}^{\ess}(\cV)>0$.
\end{itemize}
\end{corollary}

%%%%%%%%%%%%%%%%%%%%%%%%%%%%%%%%%%%%%%%%%%%%%%%%%%%%%%%%%%%%
\section{Upper bounds via the isoperimetric constant}\label{ss:III.04}
%%%%%%%%%%%%%%%%%%%%%%%%%%%%%%%%%%%%%%%%%%%%%%%%%%%%%%%%%%%%

It is possible to use the isoperimetric constants to estimate $\lambda_0(\bH)$ and $\lambda_0^{\ess}(\bH)$ from above, however, for this we need to impose additional restrictions on the metric graph.

\begin{lemma}\label{lem:est03}
	Suppose that $\ell_\ast(\cG)=\inf_{e\in\cE}|e|>0$. Then 
	\begin{align}\label{eq:est03}
	\lambda_0( \bH) & \leq \frac{\pi^2}{2\, \ell_\ast(\cG)}\alpha(\cG), &  
	\lambda_0^{\ess}(\bH) \leq \frac{\pi^2}{2\, \ell_{\ast}^{ \ess} (\cG)}\alpha_{\ess}(\cG).
	\end{align}
\end{lemma}

\begin{proof}
 To estimate $\lambda_0(\bH)$, choose any $\phi \in H^1([0,1])$ with $\phi(0) =0$, $\phi(1)=1$ and $\|\phi\|_{L^2(0,1)}=1$ and set
 \[
 \wt{\phi(x)}:= \id_{[0,1/2]}(x)\phi(2x) + \id_{(1/2,1]}(x)\phi(2-2x),\qquad x\in[0,1].
 \] 
 Assume a subgraph $\cG_0 \in \cK_\cG$ and a finite, connected subgraph $\wt \cG = (\wt \cV, \wt \cE)$ of ${\cG \setminus \cG_0}$. Then define $g \in \wt H^1_c(\cG \setminus \cG_0)$ by setting
	\begin{align*}
	g(x_e) := \begin{cases} 
	0, &\qquad e \in  \cE_{\cG\setminus \cG_0}, \; e \notin \wt \cE\\[1mm]
	1, &\qquad e \in \wt\cE_0\\[1mm]
	\phi(\frac{|x_e-u|}{|e|}),  &\qquad e=e_{u,\ti{u}}\in \wt\cE_1, u \in \partial \wt \cG \\[2mm]
	\wt\phi(\frac{|x_e- e_o|}{|e|}), &\qquad e \in \wt\cE_2 
	\end{cases}\ ,
	\end{align*}
	where $\wt \cE_0$, $\wt \cE_1$, $\wt \cE_2$ are defined as in the previous subsection and $|x_e - y|$ denotes the distance between $x_e\in e$ and some $y\in e$. If $\cG_0 \neq \varnothing$ and $v \in \cG \setminus \cG_0$ is a vertex with at least one incident edge in $\cG_0$, then either $v$ is not in $\wt \cV$ or $v$ is a boundary vertex of $\wt \cG$. In both cases, $g$ vanishes at $v$. Therefore, $g \in \wt{H}^1(\cG \setminus \cG_0)$. 	
	Next we get
	\begin{align*}
	\|g\|_{L^2(\cG \setminus \cG_0)}^2 &= \sum_{e \in \wt \cE_0} |e| + \sum_{e \in \wt \cE_1} |e| \| \phi\|^2_{L^2(0,1)} + \sum_{e \in \wt \cE_2} 2 \frac{|e|}{2} \|  \phi\|^2_{L^2(0,1)} = \mes(\wt \cG),
	\end{align*}
	and, in view of \eqref{eq:IV.06},
	\begin{align*}
	\| g' \|_{L^2(\cG  \setminus \cG_0)}^2 &= \sum_{e \in \wt \cE_1} \frac{1}{|e|} \|\phi'\|^2_{L^2(0,1)}  + \sum_{e \in \wt \cE_2}    \frac{4}{|e|} \|\phi'\|^2_{L^2(0,1)} \\
	& \leq \frac{ \|\phi' \|_{L^2(0,1)}^2}{\ell_\ast(\cG \setminus \cG_0)} (\#( \wt \cE_1 ) + 4 \#( \wt \cE_2) ) \leq \frac{2 \|\phi' \|_{L^2(0,1)}^2 }{\ell_\ast( \cG \setminus \cG_0)} \deg(\partial_\cG  \wt \cG).
	\end{align*}
	Choosing $\phi(x) = \sqrt{2}\sin(\frac{\pi}{2}x)$, we obtain the estimate
	\[
		\frac{ \| g' \|_{L^2(\cG  \setminus \cG_0)}^2 }{\|g\|_{L^2(\cG \setminus \cG_0)}^2 } \; \leq \; \frac{\pi^2}{2 \; \ell_\ast( \cG \setminus \cG_0  ) } \frac{  \deg(\partial_\cG  \wt \cG)  }{   \mes(\wt \cG)      }.
	\]
Choosing $\cG_0 = \varnothing$, \eqref{eq:Rayleigh} and \eqref{eq:Cheeger} imply the first inequality in \eqref{eq:est03}. 
	Now assume $\cG_0 \neq \varnothing$. Then
	\[
			\inf_{\substack{f\in \wt{H}^1(\cG\setminus \cG_0 )\\ f\neq0}}\frac{\|f'\|^2_{L^2(\cG \setminus \cG_0 )}}{\|f\|^2_{L^2(\cG\setminus \cG_0)}} \; \leq \; \frac{\pi^2}{2 \;  \ell_\ast( \cG \setminus \cG_0)  } \alpha(\cG \setminus \cG_0).
	\]
	Finally, using \eqref{eq:lambdaesslimit} and \eqref{eq:alphaesslimit}  we end up with  
	\[
		\lambda_0^{\ess}(\bH)  \leq \; \lim_{\cG_0 \in \cK_\cG} \frac{\pi^2}{2 \; \ell_\ast( \cG \setminus \cG_0  )} \alpha(\cG \setminus \cG_0) = \frac{\pi^2}{2 \; \ell_{\ast}^{\ess}  ( \cG )}\alpha_{\ess}(\cG).  \qedhere
	\]
\end{proof}

Combining Lemma \ref{lem:est03} with the Cheeger-type bounds \eqref{eq:Cheeger} and the estimates \eqref{eq:A=Acomb01}--\eqref{eq:A=Acomb02} and taking into account Lemma \ref{lem:equivalence}, we immediately get the following result.

\begin{corollary}\label{cor:buser}
If $\ell_\ast(\cG)>0$ and $\ell^\ast(\cG) < \infty$, then the following are equivalent:
\begin{itemize}
\item[(i)] $\lambda_0(\bH)>0$,
\item[(ii)] $\lambda_0^{\ess}(\bH)>0$,
\item[(iii)] $\alpha_{\comb}(\cG)>0$,
\item[(iv)] $\alpha^{\ess}_{\comb}(\cG)>0$.
\end{itemize}
\end{corollary}

\begin{remark}
A few remarks are in order:
\begin{itemize}
\item[(i)] If $\ell_\ast(\cG)=0$, then the estimate in \eqref{eq:est03} becomes trivial.
\item[(ii)]  Notice that \eqref{eq:est03} is better than \eqref{eq:est01} only if the isoperimetric constant satisfies
\[
\alpha(\cG) < \frac{2\,\ell_\ast(\cG)}{\ell^\ast(\cG)^2}. 
\]
\item[(iii)] In \cite{bus}, Buser noticed that the isoperimetric constant can be used for obtaining upper estimates on the spectral gap for Laplacians on compact Riemannian manifolds. Hence estimates of the type \eqref{eq:est03}   are often called Buser-type estimates. Let us mention that for combinatorial Laplacians a Buser-type estimate was first proved in \cite{am} (see also \cite{cdv, dsv}). For finite quantum graphs, a Buser-type bound can be found in \cite[Proposition 0.3]{km16}, which is, however, different from our estimate \eqref{eq:est03}.
\end{itemize}
\end{remark}

%%%%%%%%%%%%%%%%%%%%%%%%%%%%%%%%%%%%%%%%%%%%%%%%%%%%%%%%%%%%
\section{Bounds by curvature}\label{ss:III.06}
%%%%%%%%%%%%%%%%%%%%%%%%%%%%%%%%%%%%%%%%%%%%%%%%%%%%%%%%%%%%

Despite the combinatorial nature of isoperimetric constants \eqref{eq:alphaG} and \eqref{eq:alphaessG}, it is known that computation of the combinatorial isoperimetric constant \eqref{eq:alpha_comb} is an NP-hard problem (see \cite{hoo,kai,moh89} for further details). 
Our next aim is to introduce a quantity, which provides estimates for $\alpha(\cG)$ and $\alpha_{\ess}(\cG)$ and also turns out to be very useful in many situations (see Section \ref{sec:Examples}). 

Suppose now that our graph is oriented, that is, every edge is assigned a direction. For every $v\in \cV$, let $\cE_{v}^+$ and $\cE_{v}^-$ be the sets of outgoing and incoming edges, respectively. Next define the function $\rK\colon\cV\to \R\cup\{-\infty\}$ by
\be\label{eq:def_curv}
\rK\colon v\mapsto \frac{\#(\cE_v^+) - \#(\cE_v^-)}{\#(\cE_v^+)}\inf_{e\in\cE_v^+}\frac{1}{|e|}.		 
\ee
Note that $\rK$ can take both positive and negative values, and $\rK(v)= - \infty$ whenever $\#(\cE_v^+)=\emptyset$.

 \begin{lemma}\label{lem:curv}
 	Assume $\cG$ is an oriented graph such that the function $\rK$ is positive. Then the isoperimetric constant \eqref{eq:alphaG} satisfies
\be\label{eq:alpha_curv}
 \alpha(\cG) \geq \rK(\cG):=\inf_{v \in \cV} \rK(v) \ge 0 . 
\ee
 \end{lemma}
 
\begin{proof}
Let $\wt \cG\in\cK_\cG$ be a finite and connected subgraph. For every $v\in\wt\cV$, denote by $\cE_{v}^+(\wt\cG)$ and $\cE_{v}^-(\wt\cG)$ the sets of outgoing and incoming edges in $\wt\cG$. Since $\rK(v)>0$ is positive, we get
\[
\sup_{e\in \cE^+_v} |e| \le \frac{1}{\rK(v)} \Big( 1 - \frac{\#(\cE_v^-)}{\#(\cE_v^+)}\Big),
\]
for all $v\in\cV$. 
Therefore,
	\begin{align*}
		\mes(\wt \cG) = \sum_{e \in \wt \cE} |e| = & \sum_{v \in \wt \cV}\, \sum_{e \in \cE^+_v(\wt\cG)}|e| 
		\leq \frac{1}{\rK(\cG)} \sum_{v \in \wt \cV}\, \sum_{e \in \cE^+_v(\wt\cG)} 1- \frac{\#(\cE_v^-)}{\#(\cE_v^+)} \\
&	=  \frac{1}{\rK(\cG)}\sum_{v \in \wt \cV} \#(\cE_{v}^+(\wt\cG))\Big(1  -  \frac{\#(\cE_v^-)}{\#(\cE_v^+)}\Big). 
	\end{align*}
First observe that
\[
\sum_{v \in \wt \cV} \#(\cE_{v}^+(\wt\cG)) = \sum_{v \in \wt \cV} \#(\cE_{v}^-(\wt\cG)) = \#(\wt\cE).
\]	
Moreover, for any non-boundary point $v \in \wt \cV\setminus \partial_\cG\wt\cG$, the whole star $\cE_v$ is contained in $\wt \cG$ and hence $\cE_{v}^\pm(\wt\cG) = \cE_v^\pm$. Therefore, we get
\begin{align*}
\sum_{v \in \wt \cV} \#(\cE_{v}^+(\wt\cG)) \Big(1  -  \frac{\#(\cE_v^-)}{\#(\cE_v^+)}\Big) 
&= \sum_{v \in \wt \cV} \#(\cE_{v}^+(\wt\cG))  - \sum_{v \in \wt \cV} \#(\cE_{v}^+(\wt\cG)) \frac{\#(\cE_v^-)}{\#(\cE_v^+)}\\
&= \sum_{v \in \wt \cV} \#(\cE_{v}^-(\wt\cG))  - \sum_{v \in \wt \cV} \#(\cE_{v}^+(\wt\cG)) \frac{\#(\cE_v^-)}{\#(\cE_v^+)}\\
&= \sum_{v \in  \partial_{\cG}\wt\cG}\#(\cE_{v}^-(\wt\cG))  -  \#(\cE_{v}^+(\wt\cG)) \frac{\#(\cE_v^-)}{\#(\cE_v^+)} \\
&\le \sum_{v \in  \partial_{\cG} \wt{\cG} }\deg_{\wt \cG} (v) = \deg( \partial_\cG \wt \cG).
\end{align*}
Combining this with the previous estimates, we end up with the following bound
	\begin{align*}
		\mes(\wt \cG) \le \frac{1}{\rK( \cG)} \deg( \partial_\cG \wt \cG),
	\end{align*}
	which proves the claim.
	\end{proof}

\begin{remark}
The function $\rK$ is sometimes interpreted as {\em curvature}. 
Several notions of curvature have been introduced for discrete and combinatorial Laplacians. Perhaps, the closest one to \eqref{eq:def_curv} have been introduced in \cite{klw}. Namely, since the natural path metric $\varrho_0$ is intrinsic, define the function $\rK_d\colon \cV\to\R$ by
\be\label{eq:curv_discr}
\rK_d\colon v\mapsto \frac{\#(\cE^+_{v}) - \#(\cE^-_{v})}{m(v)}.
\ee   
Moreover, $m(v) = \deg(v)$ for all $v\in \cV$ if the corresponding metric graph is equilateral (i.e., $|e|\equiv 1$), and hence \eqref{eq:curv_discr} coincides with the  definition suggested for combinatorial Laplacians in \cite{dk}. Notice that for equilateral graphs \eqref{eq:def_curv} reads
\be\label{eq:curv_comb}
\rK(v) = \rK_{\rm comb}(v) := 1 - \frac{\#(\cE_v^-)}{\#(\cE_v^+)},\quad v\in\cV,
\ee
and hence in this case
\be
\frac{2}{\rK(v)} = \frac{2}{\rK_{\rm comb}(v)} = 1 + \frac{1}{\rK_d(v)},\quad v\in\cV.
\ee
It seems there is no nice connection between $\rK$ and $\rK_d$ in the general case. 
\end{remark}

\begin{remark}\label{rem:discr_estcurv}
Let us also mention that Lemma \ref{lem:curv} can be seen as the analog of \cite[Theorem 6.2]{bkw15}, where the following bound for the discrete isoperimetric constant was established:
\be\label{eq:alpha_curvdiscr}
\alpha_d(\cV) \ge \rK_d(\cV) := \inf_{v\in\cV} \rK_d(v),
\ee
{\em if $\rK_d$ is nonnegative on $\cV$.} Combining \eqref{eq:alpha_curvdiscr} with the second inequality in \eqref{eq:connCheeger}, we end up with the following bound
\begin{equation} \label{eq:bkwest}
\frac{2}{\alpha(\cG)} \le \frac{1}{\rK_d(\cV)} + \ell^\ast(\cG).
\end{equation}
\end{remark}
	
In what follows we shall call the function $\rK_{\rm comb}\colon\cV \to \Q\cup\{- \infty\}$ defined by \eqref{eq:curv_comb} as the {\em combinatorial curvature} (in \cite[p.\ 32]{dk}, $\rK_{\rm d}$ is called a {\em curvature of the combinatorial distance spheres}). Note that the curvature can take both positive and negative values, and $\rK_{\rm comb}(v)=- \infty$ whenever $\#(\cE_v^+)=\emptyset$.	The next simple estimate turns out to be very useful in applications. 
	
\begin{lemma} \label{lem:combcurv}
	Assume $\rK_{\text{comb}}$ is positive on $\cV$ and 
	\[
	\rK_{\rm comb}(\cV):= \inf_{v\in\cV}\rK_{\rm comb}(v).
	\] 
Then the isoperimetric constant \eqref{eq:alphaG} satisfies
\be\label{eq:est_curvcomb}
 \alpha(\cG) \ge  \frac{\rK_{\rm comb}(\cV)}{\ell^\ast(\cG)}. 
\ee
\end{lemma}

\begin{proof}
Noting that $\rK_{\rm comb}$ is positive and comparing \eqref{eq:curv_comb} with \eqref{eq:def_curv}, we get
\be
\frac{\rK_{\rm comb}(v)}{\ell^\ast(\cG)} \le \rK(v)
\ee
for all $v\in\cV$. Hence the claim follows from Lemma \ref{lem:curv}. 
\end{proof}	

With a little extra effort and using an argument similar to that in the proof of \eqref{eq:connCheeger} one can show the following bounds.

\begin{lemma}\label{lem:curv_combess}
Assume $\cG$ is an oriented graph such that the function $\rK$ (and hence $\rK_{\comb}$) is positive on $\cV$ and set 
\begin{align}
\rK^{\ess}(\cG) & :=\liminf_{v \in \cV} \rK(v), & \rK_{\rm comb}^{\ess}(\cV) & :=\liminf_{v\in\cV}\rK_{\rm comb}(v).
\end{align}
 Then the isoperimetric constant at infinity \eqref{eq:alphaessG} satisfies
\be
\alpha_{\ess}(\cG) \geq \rK^{\ess}(\cG), 
\ee
and
\be\label{eq:est_curvcomb}
 \frac{\rK_{\rm comb}^{\ess}(\cV)}{\ell^\ast_{\ess}(\cG)}  \leq  \alpha_{\ess}(\cG)   \leq  \frac{2}{\ell^\ast_{\ess}(\cG)}, 
\ee
\end{lemma}

Combining Lemma \ref{lem:curv_combess} with the Cheeger-type estimate, we immediately get the following result.

\begin{corollary}\label{cor:curv_combess}
		If $\cG$ is an oriented graph such that the function $\rK_{\rm comb}$ is nonnegative on $\cV$, then
\begin{align}
 \lambda_0(\bH) &\ge  \frac{\rK_{\rm comb}(\cV)^2}{4\,\ell^\ast(\cG)^2}, & \lambda_0^{\ess}(\bH)\ge  \frac{\rK_{\rm comb}^{\ess}(\cV)^2}{4\,\ell^\ast_{\ess}(\cG)^2}.
\end{align}
	In particular, if $\rK_{\rm comb}^{\ess}(\cV)>0$, then the spectrum of $\bH$ is purely discrete precisely when $\ell^\ast_{\ess}(\cG)=0$. 
\end{corollary}

\begin{remark}
Let us mention that in the case when $\rK_{\rm comb}^{\ess}(\cV)=0$ the condition $\ell^\ast_{\ess}(\cG)=0$ is no longer sufficient for the discreteness. For further details we refer to Section \ref{ss:IV.04} and, more specifically, to the example of polynomially growing antitrees (see Example \ref{ex:polyat}).
\end{remark}

%%%%%%%%%%%%%%%%%%%%%%%%%%%%%%%%%%%%%%%%%%%%%%%%%%%%%%%%%%%%
\section{Growth volume estimates}\label{ss:III.05} 
%%%%%%%%%%%%%%%%%%%%%%%%%%%%%%%%%%%%%%%%%%%%%%%%%%%%%%%%%%%%

Here we plan to exploit the results from \cite{stu} to get upper bounds on the spectra of quantum graphs in terms of the exponential volume growth rates, the so-called Brooks-type estimates (cf. \cite{bro}, \cite{hkw13}, \cite{stu} for further details and references). Following \cite{stu}, we introduce the following notation. For every $x\in\cG$ and $r>0$, let
\be\label{eq:ball}
B_r(x): = \{y\in \cG|\ \varrho_0(x,y)<r\}.
\ee
Here $\varrho_0$ is the natural path metric on $\cG$. Let also 
\be\label{eq:volball}
\vol_x(r) := \mes(B_r(x)),
\ee
and
\be\label{eq:vol*}
\vol_\ast(r) := \inf_{x\in\cG} \frac{\mes(B_r(x))}{\mes(B_1(x))}.
\ee
Next we define the following numbers
\begin{align}\label{eq:mudef}
\mu_x(\cG) & := \liminf_{r\to \infty} \frac{\log(\vol_x(r))}{r}, & \mu_\ast(\cG) & := \liminf_{r\to \infty} \frac{\log(\vol_\ast(r))}{r}.
\end{align} 
Notice that $\mu_x(\cG)$ does not depend on $x\in\cG$ if $\cG = \cup_{r>0}B_r(x)$ for some (and hence for all) $x\in\cG$.
If both conditions are satisfied, then we shall write $\mu(\cG)$ instead of $\mu_x(\cG)$. 

\begin{theorem}\label{th:brooks}
Suppose $(\cV,\varrho_0)$ is complete as a metric space. Then 
\begin{align}\label{eq:est04}
\lambda_0(\bH)  \le \lambda_0^{\ess}(\bH) \le \frac{1}{4}\mu_\ast(\cG)^2 \le \frac{1}{4}\mu(\cG)^2.
\end{align}
\end{theorem}

\begin{proof}
The first and the last inequalities in \eqref{eq:est04} are obvious and hence it remains to show that
\[
\lambda_0^{\ess}(\bH) \le \frac{1}{4}\mu_\ast(\cG)^2.
\]
Notice that by Corollary \ref{cor:gaffney}, the pre-minimal operator $\bH_0$ is essentially self-adjoint and hence $\bH$ is its closure. Let us consider the corresponding quadratic form $\gt_\cG$ defined as the closure in $L^2(\cG)$ of the form $\gt_\cG^0$ (see \eqref{eq:gt0} and \eqref{eq:gt00}). It is not difficult to check that the form $\gt_\cG$ is a  strongly local regular Dirichlet form (see \cite{fuk10} for definitions). On the other hand, using the Hopf--Rinow type theorem for graphs (see \cite{hkmw13}), with a little work one can show that
every ball $B_r(x)$ is relatively compact if $(\cV,\varrho_0)$ is complete. Therefore, by \cite[Theorem 5]{stu} and \cite[Theorem 1]{not}, \cite[Theorem 1.1]{hkw13}, we get
\begin{align*}
\lambda_0(\bH) & \le \frac{1}{4}\mu_\ast(\cG)^2, & \lambda_0^{\ess}(\bH) & \le \frac{1}{4}\mu(\cG)^2.
\end{align*}
Noting that $\mes(B_1(x))\ge 1$ for all $x\in \cG$ and taking into account \cite[Remark (e) on p.885]{hkw13}, we arrive at the desired estimate.
\end{proof}

The next result is straightforward from Theorem \ref{th:brooks}.

\begin{corollary}\label{cor:lam=0}
Let $(\cV,\varrho_0)$ be complete as a metric space. Then:
\begin{itemize}
\item[(i)] $\lambda_0(\bH)=\lambda_0^{\ess}(\bH)=0$ if $\mu(\cG)=0$. 
\item[(ii)] The spectrum of $\bH$ is not discrete if $\mu_\ast (\cG)<\infty$. 
\end{itemize}
\end{corollary}

\begin{remark}\label{rem:7.1}
Clearly, to compute or estimate $\mu_\ast(\cG)$ is a much more involved problem comparing to that of $\mu(\cG)$. However, it might happen that $\mu_\ast(\cG)<\mu(\cG)$ and hence $\mu_\ast(\cG)$ provides a better bound (see Example \ref{ex:sparse}).
\end{remark}

\begin{remark}
Let us mention that these results have several further consequences for the heat semigroup $\E^{-t\bH}$ generated by the operator $\bH$. For example, $\mu_\ast(\cG)=0$ implies the exponential instability of the corresponding heat semigroup on $L^p(\cG)$ for all $p\in [1,\infty]$ (see \cite[Corollary 2]{stu}). 
\end{remark}

We finish this section with comparing the estimates \eqref{eq:est04} with the ones obtained in \cite{ekmn} in terms of the volume growth of the corresponding discrete graph. Following \cite{hkw13} (see also \cite[\S 4.3]{ekmn}), define the constant 
\be\label{eq:mu-d}
\mu_d(\cG):=\liminf_{r\to \infty} \frac{\log m(B_r(v))}{r}
\ee
for a fixed $v\in\cV$. 
Here 
\[
m(B_r(v)) = \sum_{u\in B_r(v)} m(u),\qquad v\in\cV. 
\]
Notice that $\mu_d(\cG)$ does not depend on the choice of $v\in\cV$ if $\cG=\cup_{r>0} B_r(x)$. 

\begin{lemma}\label{lem:mu=mu}
If $\ell^\ast(\cG)<\infty$ and $(\cV,\varrho_0)$ is complete as a metric space, then 
\begin{align}\label{eq:mu=mu}
\mu(\cG)  = \mu_d(\cG). 
\end{align}
\end{lemma}

\begin{proof}
First observe that 
\[
m(B_r(v)) = 2\sum_{\{u,\ti{u}\} \subset B_r(v)} |e_{u,\ti{u}}| + \sum_{\substack{\{u,\ti{u}\}\not\subset B_r(v)\\ \{u,\ti{u}\}\cap B_r(v) \neq\varnothing}}|e_{u,\ti{u}}| \ge \mes(B_r(v)) = \vol_v(r).
\] 
for all $v\in \cV$ and $r>0$, which immediately implies $\mu(\cG) \le \mu_d(\cG)$. Similarly, we also get
\begin{equation}
m(B_r(v)) \le 2\mes(B_{r+\ell^\ast}(v))
\end{equation}
for all $v\in \cV$ and $r>0$ and hence
\[
\mu_d(\cG) \le \liminf_{r\to \infty} \frac{\log (2\vol_v(r+\ell^\ast))}{r} = \mu(\cG),
\]
which finishes the proof of \eqref{eq:mu=mu}.
\end{proof}

\begin{remark}
A few remarks are in order.
\begin{itemize}
\item[(i)] On the one hand, it does not look too surprising that the exponential growth rates for two Dirichlet forms $\gt_\cG$ and $\gt_{\rh}$ coincide. In particular this reflects the equivalence \eqref{eq:equivA} in the case of sub-exponential growth rates. However, comparing \eqref{eq:mu=mu} with the fact that there is no equality between $\lambda_0(\bH)$ and $\lambda_0(\rh)$ (see Section \ref{ss:III.00}), one can conclude that in the case of an exponential growth of volume balls, \eqref{eq:est04} might not lead to qualified estimates (and examples of trees and antitrees in the next section confirm this observation).   
\item[(ii)] Combining \eqref{eq:mu=mu} with Corollary \ref{cor:lam=0} we  obtain Theorem 4.19 from \cite{ekmn}.
\end{itemize}
\end{remark}
 	
%%%%%%%%%%%%%%%%%%%%%%%%%%%%%%%%%%%%%%%%%%%%
%%%%%%%%%%%%%%%%%%%%%%%%%%%%%%%%%%%%%%%%%%%%%%%%%%%%%%%%%%%%
\section{Examples}\label{sec:Examples}
%%%%%%%%%%%%%%%%%%%%%%%%%%%%%%%%%%%%%%%%%%%%%%%%%%%%%%%%%%%%
%%%%%%%%%%%%%%%%%%%%%%%%%%%%%%%%%%%%%%%%%%%%%%%%%%%%%%%%%%%%

In this section we are going to apply our results to certain classes of graphs (trees, antitrees, and Cayley graphs of finitely generated groups). Let us also recall that we always assume Hypotheses \ref{hyp:locfin}--\ref{hyp:02} to be satisfied.  

%%%%%%%%%%%%%%%%%%%%%%%%%%%%%%%%%%%%%%%%%%%%%%%%%%%%%%
\subsection{Trees}\label{ss:IV.02}
%%%%%%%%%%%%%%%%%%%%%%%%%%%%%%%%%%%%%%%%%%%%%%%%%%%%%%
Let us first recall some basic notions. A connected graph without cycles is called {\em a tree}. We shall denote trees (both combinatorial and metric) by $\cT$. Notice that for any two vertices $u$, $v$ on a tree $\cT = (\cV, \cE)$ there is exactly one path $\cP$ connecting $u$ and $v$.
 A tree $\cT = (\cV, \cE)$ with a distinguished vertex $o\in\cV$ is called {\em a rooted tree} and $o$ is called {\em the root} of $\cT$. In a rooted tree the vertices can be ordered according to (combinatorial) spheres. Namely, let $d(\cdot) := d(o,\cdot)$ be  the combinatorial distance to the root $o$ and $S_n$ be the $n$-th (combinatorial) sphere, i.e., the set of vertices $v\in\cV$ with $d(v)=n$. A vertex in the $(n+1)$-th sphere, which is connected to $v$ in the $n$-th sphere, is called {\em a forward neighbor} of $v$. 
 In what follows, we define an orientation on a rooted tree according to combinatorial spheres, that is, for every edge $e$ its initial vertex belongs to the smaller combinatorial sphere. 

We begin with the following simple estimate for rooted trees. According to the choice of orientation, we get $\rK_{\rm comb}(o) = \deg(o)$ and
\[
\rK_{\rm comb}(v) = \frac{\#(\cE_{v}^+) - \#(\cE_{v}^-)}{\#(\cE_{v}^+)} = \frac{\deg(v)-2}{\deg(v) -1}
\]
for all $v\in \cV\setminus\{o\}$. Therefore, $\rK_{\rm comb}$ is nonnegative on $\cV$ if there are no loose ends, that is, $\deg(v)\neq 1$ for all $v\in\cV$. Let
\begin{align*}
\deg_\ast(\cV) & :=\inf_{v\in\cV}\deg(v), & \deg_\ast^{\ess}(\cV) & :=\liminf_{v\in\cV}\deg(v).
\end{align*}
Hence we easily get 
\begin{align*}
\rK_{\rm comb}(\cT) & = \frac{\deg_\ast(\cV)-2}{\deg_\ast(\cV) -1}, & \rK_{\rm comb}^{\ess}(\cT) & = \frac{\deg_\ast^{\ess}(\cV)-2}{\deg_\ast^{\ess}(\cV) -1},
\end{align*}
and therefore we end up with the following estimate.

\begin{lemma}\label{lem:rootcurv}
Assume $\cT$ is a rooted tree without loose ends. 
Then
\begin{align}\label{eq:treesbound}
\lambda_0(\bH) & \ge \frac{\rK_{\rm comb}(\cT)^2}{4\,\ell^\ast(\cG)^2}, & \lambda_0^{\ess}(\bH) & \ge \frac{\rK_{\rm comb}^{\ess}(\cT)^2}{4\,\ell^\ast_{\ess}(\cG)^2}.
\end{align}
In particular, $\lambda_0(\bH)>0$ if and only if $\ell^\ast(\cG)<\infty$ and the spectrum of $\bH$ is purely discrete if and only if $\ell^\ast_{\ess}(\cG)=0$.
\end{lemma}

\begin{proof}
The proof immediately follows from Corollary \ref{cor:curv_combess}, Remark \ref{rem:3.3}(i) and the fact that the combinatorial curvature admits the following bound (take also into account Hypothesis \ref{hyp:02})
\[
\frac{1}{2} \le \rK_{\rm comb}(\cT) <1.\qedhere
\]
\end{proof}

\begin{remark}\label{rem:tress01}
A few remarks are in order.
\begin{itemize}
\item[(i)] In the case of regular metric trees (these are rooted trees with an additional symmetry -- all the vertices from the same distance sphere have equal degrees as well as all the edges of the same generation are of the same length), the second claim in Lemma \ref{lem:rootcurv} was observed by M.\ Solomyak in \cite{sol04}. In fact, under Hypothesis \ref{hyp:02}, conditions (5.1) and (5.5) of \cite{sol04} hold true if and only if, respectively, $\ell^\ast(\cG)<\infty$ and $\ell^\ast_{\ess}(\cG)=0$. However, the case of the Neumann Laplacian is considered in \cite{sol04}, and it follows that criteria for the positivity and discreteness for the Neumann and Dirichlet Laplacians coincide.
\item[(ii)] Let us mention that the positivity (however, without estimates) of a combinatorial isoperimetric constant for the type of trees considered in Lemma \ref{lem:rootcurv} is known (see \cite[Theorem 10.9]{woe})
\end{itemize} 
\end{remark}

In the case of trees the estimates \eqref{eq:treesbound} can be improved, however, instead of providing these generalizations we are going to consider only one particular case.
 
\begin{example}[Bethe lattices] \label{ex:tree01}
Fix $\beta \in \Z_{\ge 3}$ and consider the combinatorial graph, which is a rooted tree such that all vertices have degree $\beta$. This type of trees is called {\em Bethe lattices} (also known as {\em Cayley trees} or {\em homogeneous trees}) and they will be denoted by $\T_\beta$. Suppose that the corresponding metric graph is equilateral, that is, $|e|=1$ for all $e\in\cE$. By abusing the notation, we shall denote the corresponding metric graph by $\T_\beta$ too. 
 Then one computes
\[
\rK_{\rm comb}(\T_\beta) = \rK^{\ess}_{\rm comb}(\T_\beta) = \frac{\beta-2}{\beta-1} = :\rK_{\beta}.
\]
Noting that $\rK_{\beta} \in [1/2,1)$ and applying Lemma \ref{lem:rootcurv}, we arrive at the following estimate
\be\label{eq:bethe}
\lambda_0^{\ess}(\T_\beta)\ge \lambda_0(\T_\beta)\ge \frac{1}{4}\rK_\beta^2.
\ee

On the other hand, it is straightforward to check that (see, e.g., \cite{dk})
	\begin{align}\label{eq:alphaBethe}
			\alpha( \T_\beta) & = \rK_{\rm comb}(\T_\beta) = \frac{\beta-2}{\beta-1}, & \alpha_d( \T_\beta) & =  \frac{\beta-2}{\beta}.
	\end{align}
In particular, this implies that the equality holds in the second inequality in \eqref{eq:connCheeger}. 	
Moreover, the spectra of both operators $\bH$ and $\rh$ can be computed explicitly (see, e.g., \cite[Example 6.3]{sol04} or \cite[Theorem 1.14]{dk} together with Theorem \ref{th:below}) and, in particular,
\[
\lambda_0(\bH) = \lambda_0^{\ess}(\bH) = \arccos^2\Big(\frac{2\sqrt{\beta-1}}{\beta}\Big).
\]	
Comparing the last equality with the estimate \eqref{eq:bethe}, one can notice a gap between these estimates.  

Let us mention that 
\[
\mu(\T_\beta) = \mu_o(\T_\beta) = \mu_\ast(\T_\beta) = \beta-1,
\]	
 and thus the volume growth estimates \eqref{eq:est04} do not provide a reasonable upper bound for large values of $\beta$.	\hfill$\lozenge$
\end{example}

Finally, we would like to mention that the absence of loose ends in Lemma \ref{lem:rootcurv} is essential as the next example shows.

\begin{example}[A ``sparse" tree with loose ends]\label{ex:sparse}
Consider the half-line $\R_{\ge 0}$ as an equilateral graph with vertices at the integers. Let us write $v_n$ for the vertex placed at  $n \in \Z_{\ge 0}$.  
Then, we will modify this graph by attaching edges to the vertices $v_n$ with $n\ge 1$. More precisely, to the $j^2$-th vertex $v_{j^2}$ with $j\in\Z_{\ge 1}$, we attach $2^{j^2}$ edges and to every other vertex $v_n$ with $n \notin \{j^2\}_{j\ge 1}$, we attach exactly one edge (see Figure \ref{fig:tree2}).

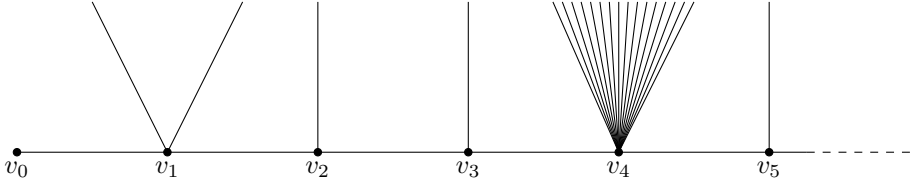
\begin{figure}
	\begin{center}
		\begin{tikzpicture}    
		
		%%% LINE
		\draw	[-]	(0,0) -- (10.5,0);
		\draw	[thin, -, dashed]	(10.5,0) -- (12,0);
		
		%%% LOOSE ENDS
		
		\draw		(2,0) -- (1,2);
		\draw		(2,0) -- (3,2);
		\draw		(4,0) -- (4,2);
		
		\foreach \x in {1,...,16}
		\draw (8,0) -- (7 + \x*2/16, 2);
		
		\draw		(6,0) -- (6,2);		
		\draw		(10,0) -- (10,2);
		
		%drawing vertices 	
		\foreach \x in {0,...,5}
		\filldraw 
		(2*\x,0) circle (1.4pt) node [below] {$v_\x$}  ;
		\end{tikzpicture}
	\end{center}
	\caption{Tree with loose ends.}\label{fig:tree2}
\end{figure}

 Clearly, we end up with a tree graph $\cT$. For simplicity, we shall assume that the corresponding metric graph is equilateral, that is, $|e|=1$ for all $e\in\cT$. This tree is in a certain sense sparse and as a result it turns out that
\[ 
\mu_\ast (\cT) = 0,
\] 
and hence, by Theorem \ref{th:brooks},
\[
\lambda_0(\bH) = \lambda_0^{\ess}(\bH) = 0.
\]
In fact, it is enough to show that $\vol_\ast(r) = 1$ for all $r>1$. Namely, take $r>1$ and set $j_r:= 1+\floor{(r+1)/2}$, where $\floor{\cdot}$ is the usual floor function. Since
$j_r^2 - (j_r-1)^2 >r$, we get
\[ 
1 \leq  \vol_\ast(r) \leq \inf_{n\ge j_r} \frac{ \mes(B_r(v_{n^2} ))  } {B_1(v_{n^2}) } = \inf_{n\ge j_r} \frac{2^{n^2} + 2r + 2(r-1) } { 2^{n^2} +2 } = 1.
\]

It is interesting to mention that in this case $\mu(\cT) = \log(2)>0$. Indeed, 
\[
2r-1 + \sum_{k=1}^{\floor{\sqrt{r}}-1} (2^{k^2}-1) \le \vol_o(r) = \mes( B_r(v_0) ) \le 2r - 1 + \sum_{k=1}^{\floor{\sqrt{r}}} (2^{k^2}-1)
\] 
and hence for all $r>1$ we get
\[
2^{(\floor{\sqrt{r}}-1)^2} < \vol_o(r) \le 2^{\floor{\sqrt{r}}^2+1},
\]
which implies the desired equality. \hfill$\lozenge$
\end{example}
	
%%%%%%%%%%%%%%%%%%%%%%%%%%%%%%%%%%%%%%%%%%%%%%%%%%%%%%%%%%%%
\subsection{Antitrees} \label{ss:IV.04}	
%%%%%%%%%%%%%%%%%%%%%%%%%%%%%%%%%%%%%%%%%%%%%%%%%%%%%%%%%%%%
Let $\cG_d = (\cV, \cE)$ be a connected combinatorial graph. Fix a root vertex $o \in \cV$ and then order the graph with respect to the combinatorial spheres $S_n$, $n \in \Z_{\ge 0}$ (notice that $S_0=\{o\}$). The connected graph $\cG_d$ is called an {\em antitree} if every vertex in $S_n$ is connected to every vertex in $S_{n+1}$ and there are no horizontal edges, i.e., there are no edges with all endpoints in the same sphere (see Figure \ref{fig:antitree}). Clearly, an antitree is uniquely determined by the sequence $s_n := \#(S_n)$, $n\in\Z_{\ge 1}$.

	\begin{figure}
	\begin{center}
		\begin{tikzpicture}    [%
		,scale=.6
		,every node/.style={scale=.6}]
		%%%%%%%%%% EDGES %%%%%%%%%%%%%%%%%%
		%%%% FROM S0 to S1
		\draw								(0,0) -- (-1, 1.5) ;
		\draw								(0,0) -- (1, 1.5) ;
		
		%%%% FROM S1 to S2
		\draw								(-1,1.5) -- (-2, 3) ;
		\draw								(-1,1.5) -- (0, 3) ;
		\draw								(-1,1.5) -- (2, 3) ;
		\draw								(1,1.5) -- (-2, 3) ;
		\draw								(1,1.5) -- (0, 3) ;
		\draw								(1,1.5) -- (2, 3) ;
		
		%%%% FROM S2 to S3
		\draw								(-2, 3) -- (-3, 4.5);
		\draw								(-2, 3) -- (-1, 4.5);
		\draw								(-2, 3) -- (1, 4.5);
		\draw								(-2, 3) -- (3, 4.5);
		\draw								(0, 3) -- (-3, 4.5);
		\draw								(0, 3) -- (-1, 4.5);
		\draw								(0, 3) -- (1, 4.5);
		\draw								(0, 3) -- (3, 4.5);
		\draw								(2, 3) -- (-3, 4.5);
		\draw								(2, 3) -- (-1, 4.5);
		\draw								(2, 3) -- (1, 4.5);
		\draw								(2, 3) -- (3, 4.5);
		
		%%%% FROM S3 to S4
		\draw [dashed, gray]			(-3, 4.5) -- (-4, 6);
		\draw [dashed , gray]			(-3, 4.5) -- (-2, 6);
		\draw [dashed, gray]			(-3, 4.5) -- (0, 6);
		\draw [dashed , gray]			(-3, 4.5) -- (2, 6);
		\draw [dashed , gray]			(-3, 4.5) -- (4, 6);
		\draw [dashed , gray]			(-1, 4.5) -- (-4, 6);
		\draw [dashed , gray]			(-1, 4.5) -- (-2, 6);
		\draw [dashed , gray]			(-1, 4.5) -- (0, 6);
		\draw [dashed , gray]			(-1, 4.5) -- (2, 6);
		\draw [dashed , gray]			(-1, 4.5) -- (4, 6);
		\draw [dashed , gray]			(1, 4.5) -- (-4, 6);
		\draw [dashed , gray]			(1, 4.5) -- (-2, 6);
		\draw [dashed , gray]			(1, 4.5) -- (0, 6);
		\draw [dashed , gray]			(1, 4.5) -- (2, 6);
		\draw [dashed , gray]			(1, 4.5) -- (4, 6);
		\draw [dashed , gray]			(3, 4.5) -- (-4, 6);
		\draw [dashed , gray]			(3, 4.5) -- (-2, 6);
		\draw [dashed , gray]			(3, 4.5) -- (0, 6);
		\draw [dashed , gray]			(3, 4.5) -- (2, 6);
		\draw [dashed , gray]			(3, 4.5) -- (4, 6);
		
		%%%% SPHERES NUMBERING
		\draw	[thin, ->, dashed]			(4, 1.5) -- (1.5, 1.5);
		\draw	[thin, ->, dashed]			(4, 3) -- (2.5, 3);
		\draw	[thin, -> , dashed]			(4, 4.5) -- (3.5, 4.5);
		\draw	[thin, -> , dashed]			(4, 0) -- (0.5, 0);
		\node  at (4.5, 0) {\Large $S_0$};
		\node  at (4.5, 1.5) {\Large $S_1$};
		\node  at (4.5, 3) {\Large $S_2$};
		\node  at (4.5, 4.5) {\Large $S_3$};	
		%%%%%%%%%%% VERTICES %%%%%%%%%%%%%
		\filldraw 
		(0,0) circle (2pt) 
		(-1, 1.5) circle (2pt)
		(1, 1.5) circle (2pt)
		(-2, 3) circle (2pt)
		(0, 3) circle (2pt)
		(2, 3) circle (2pt)
		(-3, 4.5) circle (2pt)
		(-1, 4.5) circle (2pt)
		(1, 4.5) circle (2pt)
		(3, 4.5) circle (2pt);
		\end{tikzpicture}
	\end{center}
	\caption{Example of an antitree with $s_n = n+1$.}\label{fig:antitree}
\end{figure}
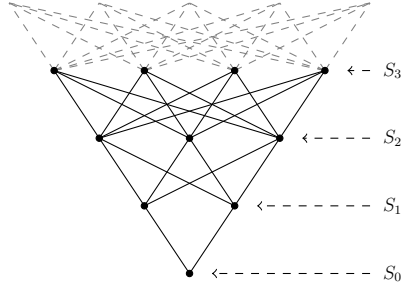

Let us denote antitrees by the letter $\cA$ and also define the edge orientation according to the combinatorial ordering, that is, for every edge $e$ its initial edge is the one in the smaller combinatorial sphere. It turns out that the curvatures of antitrees can be computed explicitly. Namely, define the following quantities:
\be\label{eq:l_n}
\ell_n := \sup_{e\in \cE_v^+\colon v\in S_n}|e|,
\ee
and 
\begin{align} \label{eq:curvatureantitree}
	\rK_0 & :=1, & \rK_{n+1} & := 1 - \frac{s_{n}}{s_{n+2}} 
\end{align}
for all $n\in\Z_{\ge 0}$.

\begin{lemma}\label{lem:curv_AT}
If $\cA$ is an antitree, then 
\begin{align}\label{eq:curvAT1}
\rK_{\rm comb}(\cA) &= \inf_{n\ge 0} {\rK_n}, & \rK_{\rm comb}^{\ess}(\cA) & = \liminf_{n\to \infty}{\rK_n}, 
\end{align}
and
\begin{align}\label{eq:curvAT2}
\rK(\cA) & = \inf_{n\ge 0} \frac{\rK_n}{\ell_n}, & \rK^{\ess}(\cA) & = \liminf_{n\to \infty}\frac{\rK_n}{\ell_n}. 
\end{align}
\end{lemma}

\begin{proof}
The proof follows by a direct inspection since $\rK_{\rm comb}(v) = \rK_n$ for all $v\in S_n$ and $n\in\Z_{\ge 0}$.
\end{proof}

Combining Lemma \ref{lem:curv_AT} with the estimates for the corresponding isoperimetric constants (e.g., Corollary \ref{cor:curv_combess}), we immediately end up with the estimates for $\lambda_0(\bH)$ and $\lambda_0^{\ess}(\bH)$. Let us demonstrate this by considering two examples.

\begin{example}[Exponentially growing antitrees]\label{ex:AT1}
Fix $\beta \in \Z_{\ge 2}$ and let $\cA_\beta$ be an antitree with sphere numbers $s_n = \beta^n$. Then $\rK_0 = 1$ and
\be
\rK_n = 1 - \beta^{-2}
\ee
for all $n\in \Z_{\ge 1}$. Hence by Lemma \ref{lem:curv_AT}
\[
\frac{1 - \beta^{-2}}{\ell^\ast(\cA_\beta)}\le \rK(\cA_\beta) \le \frac{1}{\ell^\ast(\cA_\beta)} 
\]
and
\[
\rK^{\ess}(\cA_\beta) = \frac{1 - \beta^{-2}}{\ell_{\ess}^\ast(\cA_\beta)}. 
\]
Applying Lemmas \ref{lem:curv} and \ref{lem:curv_combess} together with Theorem \ref{th:Cheeger} and Lemma \ref{lem:est01}, we get
\be
\frac{(1 - \beta^{-2})^2}{4\,\ell^\ast(\cA_\beta)^2} \le \lambda_0(\bH_\beta) \le \frac{\pi^2}{\ell^\ast(\cA_\beta)^2}, 
\ee
and
\be
\frac{(1 - \beta^{-2})^2}{4\,\ell^\ast_{\ess}(\cA_\beta)^2} \le \lambda_0^{\ess}(\bH_\beta) \le \frac{\pi^2}{\ell^\ast_{\ess}(\cA_\beta)^2}.
\ee
In particular, these bounds imply that the Kirchhoff Laplacian $\bH_\beta$ is uniformly positive if and only if $\ell^\ast(\cA_\beta)<\infty$. Moreover, its spectrum is purely discrete exactly when $\ell^\ast_{\ess}(\cA_\beta) = 0$ (cf. Corollary \ref{cor:curv_combess}).

Finally, let us compare these estimates with the volume growth estimates under the assumption that the tree is equilateral. In this case,
\[
\rK(\cA_\beta) = \rK^{\ess}(\cA_\beta) = 1-\beta^{-2}.
\]
On the other hand, 
\[
\mes(B_n(o)) = \sum_{k=0}^{n-1} \beta^{2k+1} = \beta\frac{\beta^{2n} - 1}{\beta^2-1},
\]
and then \eqref{eq:mudef} implies that $\mu(\cA_\beta) = 2\log(\beta)$. With a little more work one can show that 
\[
\mu_\ast(\cA_\beta) = \mu(\cA_\beta) = 2\log(\beta).
\]
 Indeed, it suffices to note that $\mu_\ast(\cA_\beta) \le \mu(\cA_\beta)$. Moreover, for all $x\in e_{u,v}$ where $e$ connects $S_n$ with $S_{n+1}$,  $n\in\Z_{\ge 0}$ we have 
 \[
 \mes(B_1(x)) \le \mes(B_1(v)) = \beta^{n}+\beta^{n+2} = \beta^n(\beta^2+1)
 \]
  and for all $r>2$
\begin{align*}
	\mes(B_r(x)) & \geq \mes( B_{\floor{r}} (u)) 
	= \mes( B_{n+\floor{r}} (o)) - \mes( B_{n-\floor{r}}(o)) \\
	&\ge \mes( B_{n+\floor{r}} (o)) - \mes( B_{n}(o)) = \sum_{k=n}^{n+\floor{r}-1} \beta^{2k+1} = \beta^{2n+1}\frac{\beta^{2\floor{r}} - 1}{\beta^2-1}.
\end{align*}
Thus, we obtain 
\begin{align*} 
\vol_\ast(r) = \inf_{x \in \cG} \frac{ \mes(B_r(x))}{\mes(B_1(x))} \geq \inf_{n{\ge 0}} \frac{\beta^{2n+1}\frac{\beta^{2\floor{r}} - 1}{\beta^2-1}}{\beta^n(\beta^2+1)} =  
\frac{\beta^{2\floor{r}+1} - \beta}{\beta^4-1},
\end{align*}
which shows that $\mu_\ast(\cA_\beta) \ge 2\log(\beta)$ and hence we are done. 

Notice that the volume growth estimates \eqref{eq:est04} do not provide a reasonable upper bound for large values of $\beta$. \hfill$\lozenge$
\end{example}

\begin{example}  [Polynomially growing antitrees] \label{ex:polyat}
Fix $q \in \Z_{>0}$ and let $\cA^q$ be the antitree with sphere numbers $s_n = (n+1)^q$, $n\ge 0$ (the case $q=1$ is depicted on Figure \ref{fig:antitree}). Then 
\be
\rK_n = 1 - \frac{n^{q}}{(n+2)^{q}} = 1 - \Big(\frac{n}{n+2}\Big)^q = \frac{2q}{n} + \OO(n^{-2}),
\ee
as $n\to \infty$. Hence, by Lemma \ref{lem:curv_AT},
\[
\rK_{\rm comb}(\cA^q) =\rK_{\rm comb}^{\ess}(\cA^q) = 0
\]
and
\begin{align*}
\rK(\cA^q) & = \inf_{n\ge 0} \frac{1}{\ell_n}\left( 1 - \Big(\frac{n}{n+2}\Big)^q \right), & \rK^{\ess}(\cA^q) & =  \liminf_{n\to \infty} \frac{1}{\ell_n}\left( 1 - \Big(\frac{n}{n+2}\Big)^q \right).
\end{align*}

Clearly, further analysis heavily depends on the behavior of the sequence $\{\ell_n\}$. Let us consider one particular case. Fix an $s\ge0$ and assume now that 
\[
|e| = (n+1)^{-s}
\]
for each edge $e$ connecting $S_n$ and $S_{n+1}$. Let us denote the corresponding Kirchhoff Laplacian by $\bH_{q,s}$. It is not difficult to show by applying Theorem \ref{th:sa} that the corresponding pre-minimal operator is essentially self-adjoint whenever $s \leq q+1$, however, $(\cV_q,\varrho_0)$ is complete exactly when $s\in[0,1]$. 

\begin{remark}
In our forthcoming publication we shall show that the pre-minimal operator $\bH_{0}$ is essentially self-adjoint exactly when the corresponding metric graph has infinite volume, that is, when $s \le 2q+1$. Moreover, in the case $s>2q+1$, the deficiency indices of $\bH_0$ are equal to $1$ and one can describe all self-adjoint extensions of $\bH_0$.
\end{remark}
 	
Since $\ell_n=(n+1)^{-s}$ for all $n\in\Z_{\ge 0}$, we get
\begin{align*}
\ell^\ast(\cA^q) & = 1, & \ell^\ast_{\ess}(\cA^q) & = \begin{cases} 1, & s=0\\ 0, & s>0\end{cases},
\end{align*}
and
\be
\rK^{\ess}(\cA^q) = \lim_{n\to \infty} (n+1)^s\left( 1 - \Big(\frac{n}{n+2}\Big)^q\right) = \begin{cases} 0, & s\in[0,1), \\ 2q, & s=1, \\ +\infty, & s>1.\end{cases}
\ee
In the case $s=1$, it is easy to show that the sequence $\{\rK_n/\ell_n\}$ is strictly increasing and hence this is also true for all $s>1$. Hence 
\[
\rK(\cA^q) = \rK(o) = 1,\quad s\ge1.
\]
Moreover, the corresponding isoperimetric constant is given by $\alpha(\cA^q) =\rK(\cA^q) = 1$ (to see this just take the ball $B_1(o)$ as a subgraph $\cG$ and then one gets $\alpha(\cA^q)\le 1$, which together with \eqref{eq:alpha_curv} implies the equality).

Next let us compute $\mu(\cA^q)$ assuming that $s\in[0,1]$ (otherwise we can't apply the result from Section \ref{ss:III.05}). Set
\[
r_n:= \sum_{k=0}^{n-1} \ell_k = \sum_{k=0}^{n-1} \frac{1}{(1+k)^s} =(1+o(1))\times \begin{cases} \frac{n^{1-s}}{1-s}, & s\in[0,1),\\[1mm]\log(n), &s=1,
\end{cases} 
\]
as $n\to \infty$. 
Then 
\[
\vol_o(r_n) = \sum_{k=0}^{n-1}\ell_k s_{k}s_{k+1} = \sum_{k=0}^{n-1}(k+1)^{q-s} (k+2)^{q} = \frac{n^{2q-s+1}}{2q-s+1}(1+o(1))
\]
as $n\to \infty$. Therefore, it is not difficult to show that
\be
\mu(\cA^q) = \mu_o(\cA^q) = \lim_{n\to \infty} \frac{\log(\vol_o(r_n))}{r_n} = \begin{cases} 0, & s\in [0,1),\\ 2q, & s=1.\end{cases}
\ee

Applying Theorem \ref{th:brooks} together with Lemma \ref{lem:curv} and Lemma \ref{lem:curv_combess}, we end up with the following estimates.

\begin{lemma}\label{lem:at_qs} 
Assume $q\in \Z_{\ge 1}$ and $s\in\R_{\ge 0}$. Then
\be
\lambda_0(\bH_{q,s}) = \lambda_0^{\ess}(\bH_{q,s}) = 0
\ee
if and only if $s\in[0,1)$. If $s\ge 1$, then the operator $\bH_{q,s}$ is uniformly positive and 
\begin{align}
\frac{1}{4}\le \lambda_0(\bH_{q,s}) & \le \pi^2, & \lambda_0^{\ess}(\bH_{q,s}) & = \begin{cases} q^2, & s=1,\\ +\infty, & s>1. \end{cases}
\end{align}
\end{lemma}

\begin{remark}
The exact value of $\lambda_0(\bH_{q,s})$ for $s\ge 1$ or at least its asymptotic behavior with respect to $q$ remains an open problem.\hfill$\lozenge$
\end{remark}
\end{example}

%%%%%%%%%%%%%%%%%%%%%%%%%%%%%%%%%%%%%%%%%%%%%%%%%%%%%%%%%%%%
\subsection{Cayley graphs}\label{ss:IV.01}
%%%%%%%%%%%%%%%%%%%%%%%%%%%%%%%%%%%%%%%%%%%%%%%%%%%%%%%%%%%%
Suppose $\Gamma$ is a finitely generated (infinite) group with the set of generators $S$. {\em The Cayley graph} $\cC(\Gamma,S)$ of $\Gamma$ with respect to $S$ is the vertex set $\Gamma$ and $u\sim v$ exactly when $u^{-1}v\in S$. This graph is connected, locally finite and regular ($\deg(v) = \#S$ for all $v\in\Gamma$). We assume that the unit element $o$ does not belong to the set $S$ (this excludes loops). The lattice $\Z^d$ is the standard example of a Cayley graph. Notice also that the Bethe lattice $\T_\beta$ is a Cayley graph if either $S=\{a_1,\dots,a_\beta|\, a_i^2=o,\ i=1,\dots,\beta\}$ or $\beta=2N$ and $\Gamma = \mathbb{F}_N$ is a free group of $N$ generators.  

It is known that the positivity of a combinatorial isoperimetric constant $\alpha_{\comb}$ is closely connected with the amenability of the group $\Gamma$ (this is a variant of F\o lner's criterion, see, e.g., \cite[Proposition 12.4]{woe}).

\begin{theorem}\label{th:amenab}
If $\cG_d=\cC(\Gamma,S)$ is the Cayley graph of a finitely generated group $\Gamma$, then $\alpha_{\comb}(\Gamma) = 0$ if and only if $\Gamma$ is an amenable group.
\end{theorem}

Notice that the class of amenable groups contains all Abelian groups, all subgroups of amenable groups, all solvable groups etc. In turn, the class of non-amenable groups includes countable discrete groups containing free subgroups of two generators. For further information on amenability and Cayley graphs we refer to \cite{mei, moh, pi, psc, wag, woe}.    

Combining Theorem \ref{th:amenab} with Corollary \ref{cor:alpha_comb} and Corollary \ref{cor:buser}, we arrive at the following result.

\begin{lemma}
Let $\cG_d$ be a Cayley graph $\cC(\Gamma,S)$ of a finitely generated group $\Gamma$. Also, let $|\cdot|\colon \cE\to \R_{>0}$ and $\cG = (\cG_d,|\cdot|)$ be a metric graph. Then:
\begin{itemize}
\item[(i)] If $\Gamma$ is non-amenable, then $\lambda_0(\bH)>0$ if and only if $\ell^\ast(\cG)<\infty$. Moreover, the spectrum of $\bH$ is purely discrete if and only if $\ell_{\ess}^\ast(\cG)=0$.
\item[(ii)] If $\Gamma$ is amenable, then $\lambda_0(\bH)=\lambda_0^{\ess}(\bH)=0$ whenever $\ell_\ast(\cG)>0$.
\end{itemize}
\end{lemma}

\begin{remark}
\begin{itemize}
\item[(i)] If $\Gamma$ is an amenable group, then the analysis of $\lambda_0(\bH)$ and $\lambda_0^{\ess}(\bH)$ in the case $\ell_\ast(\cG)=0$ remains an open (and, in our opinion, rather complicated) problem.
\item[(ii)] The volume growth provides a number of amenability criteria. For example, groups of polynomial or subexponential  growth are amenable. For further results and references we refer to \cite{psc}. 
\item[(iii)] Using a completely different approach, the inequality $\lambda_0(\bH)>0$ was proved recently in \cite[Theorem 4.16]{car} for Cayley graphs of free groups under the additional symmetry assumption that edges in the same edge orbit have the same length.
\end{itemize}
\end{remark}

%%%%%%%%%%%%%%%%%%%%%%%%%%%%%%%%%%%%%%%%%%%%%%%%%%%%%%%%%%%%%%%%%
%%%%%%%%%%%%%%%%%%%%%%%%%%%%%%%%%%%%%%%%%%%%%%%%%%%%%%%%%%%%%%%%%
\appendix

\section{Cheeger's inequality for discrete Laplacians}\label{sec:app}

Let $\cG_d = (\cV, \cE)$ be an (unoriented) graph with countably infinite sets of vertices $\cV$ and edges $\cE$. Also, assume that Hypothesis \ref{hyp:locfin} is satisfied. Let $m\colon \cV\to \R_{>0}$ and $b\colon \cV\times\cV\to \R_{\ge0}$ be weight functions such that $b(u,v)=b(v,u)$ for all $u,v\in \cV$ and $b(u,v)\neq 0$ only if $u\sim v$. In fact, $b$ can be considered as a weight function on the edge set $\cE$. Usually, the triple $(\cV,m,b)$ is called a {\em weighted graph}. With every such a triple one can associate a Laplace operator defined by the difference expression
\be\label{eq:wLap}
(\tau f)(v) := \frac{1}{m(v)}\sum_{u\sim v} b(u,v)(f(v) - f(u)), \quad v\in\cV.
\ee
Since the graph $\cG_d$ is locally finite, $\tau$ is well defined on compactly supported functions and hence gives rise to a nonnegative symmetric pre-minimal operator in $\ell^2(\cV;m)$. Let us denote its Friedrichs extension by $\rh$.

 The Cheeger inequality for $\rh$ was proved recently in \cite{bkw15} by using the notion of intrinsic metrics on graphs (see Theorem 3.1 and Theorem 3.3 in \cite{bkw15}). The main aim of this section is to give a slight improvement to this estimate. Namely, let $d\colon \cE\to \R_{>0}$ be a weight (or edge lengths). Similar to \cite{bkw15}, we shall call $d$ {\em intrinsic} on $\cG_d$ (with respect to $m$ and $b$) if the following inequality
 \begin{equation} \label{eq:intr}
	\sum_{e \in \cE_v} d(e)^2 b(e) \; \leq \; m(v) 
	\end{equation}
holds for all $v \in \cV$. 

For every $X \subseteq \cV$, we define its boundary edges by
		\begin{equation*}
		\cE_b(X) = \{e \in \cE| \; e \text{ connects } X \text{ and } \cV \backslash X \}.
		\end{equation*}
For any $U \subseteq \cV$, define
		\begin{equation} \label{eq:IsoX}
		\alpha_d (U) :=   \inf_{ \substack{X \subseteq U \\X \text{ finite}}} \; \frac{(d\cdot b)(\cE_b(X))}{m(X)},
		\end{equation}
where for $X \subseteq \cV$,
		\begin{align*}
		m(X) & = \sum_{v \in X} m(v), & (d\cdot b)(\cE_b(X)) & =  \sum_{e \in \cE_b(X)} d(e)b(e).
		\end{align*}
We define the isoperimetric constant with respect to $d$ by
		\begin{equation}\label{eq:Iso}
		\alpha := \alpha_d(\cV).
		\end{equation}
The isoperimetric constant at infinity is given by
		\begin{equation}\label{eq:IsoEss}
		\alpha_{\ess} := \sup_{ \substack{X \subseteq \cV \\X \text{ finite}}} \alpha_d ( \cV \backslash X).
		\end{equation}
		
\begin{theorem}\label{th:cheegerdiscr}
If $d$ is an intrinsic weight, then
	\begin{align} \label{eq:CheegerD}
		\lambda_0(\rh) & \geq \frac{1}{2}\alpha^2, & \lambda_0^ {\ess}(\rh) & \geq \frac{1}{2} \alpha_{\ess}^2.
	\end{align}
\end{theorem}
					
\begin{remark}					
As it was already mentioned, the Cheeger estimates for weighted graph Laplacians were proved in \cite{bkw15}. However, the definition of the isoperimetric constants in \cite{bkw15} uses {\em metrics} and hence one has to replace $d$ in \eqref{eq:IsoX} by the corresponding path metric $\varrho_d$ defined on $\cV$ in a standard way 
\be\label{eq:pathmetric_d}
\varrho_d(u,v) := 
\inf_{\cP=\{v_0,\dots,v_n\}\colon v_0=u\ v_n=v}\sum_{k} d(e_{v_{k-1},v_k}).  
\ee
Clearly, $\varrho_d$ is intrinsic (in the sense of \cite{bkw15}) if so is the weight $d$ since 
\be\label{eq:rho<d}
\varrho_d(u,v) \le d(u,v)
\ee
for all $u\sim v$.  
Of course, in certain cases this leads to the same isoperimetric constant (e.g., if $\cG_d$ is a tree), however, for graphs having a lot of cycles a construction of an intrinsic metric becomes a highly nontrivial task, which automatically implies complications in calculating the corresponding isoperimetric constant. On the other hand, to construct an intrinsic weight (in the sense of \eqref{eq:intr}) is a rather simple task, in particular, for the weighted Laplacian \eqref{eq:tau} (see Remark \ref{rem:ourLaplace}).
\end{remark} 

The proof of Theorem \ref{th:cheegerdiscr} is literally the same as of Theorem 3.1 and Theorem 3.3 from \cite{bkw15}, however, we shall give it below for the sake of completeness.  

	\begin{lemma} [Co-area formulae]\label{lem:coarea}
Let $m$ and $d$ be weight functions on $\cV$ and $\cE$, respectively. For any $f\colon \cV\to \R_{\ge 0}$ and $t\ge 0$, let $\Omega_t := \Omega_t(f) = \{v \in \cV| \; f(v) > t\}$. Then
		\begin{align}
		& \sum_{v \in \cV} f(v) m(v) = \int_0^\infty m(\Omega_t) \; dt, \label{eq:coarea1} \\
		& \sum_{e \in \cE} d(e) |f(e_i) - f(e_0)| = \int_0^\infty d(\cE_b(\Omega_t)) \; dt, \label{eq:coarea2}
		\end{align}
		where the value $+ \infty$ on both sides of the equation is allowed.
	\end{lemma}
	
	\begin{proof}
For an interval $I \subseteq \R$, let $\id_I(s)$ be its indicator function. Then
		\begin{align*}
			\sum_{v \in \cV} f(v) m(v) &= \sum_{v \in \cV} m(v) \int_0^{f(v)}  dt  = \sum_{v \in \cV} m(v) \int_0^\infty \id_{[0, f(v)) } (t) \; dt \\
			&= \int_0^\infty \sum_{v \in \cV} m(v) \id_{[0 ,f(v)) } (t)  \; dt = \int_0^\infty \sum_{v \in \Omega_t} m(v) \; dt = \int_0^\infty  m( \Omega_t) \; dt.
		\end{align*}
		For every $e \in \cE$, put $I_e := [\min\{ f(e_0), f(e_i) \}, \max\{f(e_0), f(e_i) \}) \subset \R$. We have $t \in I_e$ if and only if $e \in \cE_b(\Omega_t)$. Hence
		\begin{align*}
		 \sum_{e \in \cE} &d(e) |f(e_i) - f(e_0)|  = \sum_{e \in \cE} d(e) \int_{I_e}  dt = \sum_{e \in \cE} d(e) \int_0^\infty \id_{I_e} (t) \; dt \\
		 &= \int_0^\infty \sum_{e \in \cE} d(e) \id_{I_e} (t) \; dt  = \int_0^\infty \sum_{e \in \cE_b(\Omega_t)} d(e) \; dt = \int_0^\infty  d(\cE_b(\Omega_t)) \; dt. \qedhere
		 \end{align*}
	\end{proof}
	
\begin{proof}[Proof of Theorem \ref{th:cheegerdiscr}]
	We start by proving the first estimate in \eqref{eq:CheegerD}. The Rayleigh quotient implies that it suffices to show that
	\begin{equation} \label{eq:ineqproof}
	 2\, \gt_{\rh}[u]  \geq \alpha^2  \| u \|^2_{\ell^2(\cV, m)}
	 \end{equation}
holds for all real-valued $u$ with finite support, where 
	\[	
	\gt_{\rh}[u] = (\rh u,u)_{\ell^2(\cV,m)} = \sum_{e \in \cE} b(e) | u(e_i) - u(e_0)|^2
	\]
is the corresponding quadratic form. 
	Let us now exploit Lemma \ref{lem:coarea} with $f := u^2$. Notice that the set $\Omega_t$ is finite for all $t \geq 0$ and hence by \eqref{eq:IsoX} and \eqref{eq:Iso} we have $(d\cdot b)(\cE_b(\Omega_t)) \ge \alpha\, m(\Omega_t) $ for all $t\ge 0$. 
	 Therefore we get from the co-area formulas
	\begin{align*}
	\alpha\,  \| u \|^2_{\ell^2(\cV, m)}  & =	\alpha \sum_{v \in \cV} u(v)^2 m(v) = \alpha\int_0^\infty m(\Omega_t)\, dt \\
	& \leq \int_0^\infty (d\cdot b)(\cE_b(\Omega_t)) \; dt = \sum_{e \in \cE} d(e) b(e)|u(e_i)^2 - u(e_0)^2|  \\
	& = \sum_{e \in \cE} \sqrt{b(e)} |u(e_i) - u(e_0)|\cdot d(e)\sqrt{b(e)}|u(e_i) + u(e_0)|\\
		&\leq \gt_{\rh}[u]^{1/2} \left( \sum_{e \in \cE}  d(e)^2b(e) (u(e_i) + u(e_0))^2 \right)^{1/2}
	\end{align*}
	by employing the Cauchy--Schwarz inequality in the last step. 
	On the other hand,
	\begin{align*}
		\sum_{e \in \cE} {d(e)^2}{b(e)} (u(e_i) + u(e_0))^2 &\leq 2  \sum_{e \in \cE} {d(e)^2}{b(e)} (u(e_i)^2 + u(e_0)^2) \\
		&=  2 \sum_{v \in \cV} u(v)^2 \sum_{e \in \cE_v} {d(e)^2}{b(e)} \leq 2 \| u \|^2_{\ell^2(\cV, m)},
	\end{align*}
	where we used \eqref{eq:intr} in the last step. 
	
	To get the second inequality, assume $X \subseteq \cV$ finite. Let $P$ denote the orthogonal projection onto the subspace of functions vanishing on $X$. Then $\rh_{\cV\backslash X} := P \rh P$ with $\dom(\rh_{\cV\backslash X}) = \dom(\rh)$ is a relatively compact perturbation of $\rh$. Thus we have
	\[ \lambda_0^{\ess} (\rh) = \lambda_0^{\ess} (\rh_{\cV\backslash X}) \geq \lambda_0 (\rh_{\cV\backslash X}) = \inf_{\substack{u\neq 0 }} \frac{\gt_{\rh}[u]}{\|u\|_{\ell^2(\cV; m)}},   \]
	where the infimum is taken over all real-valued $u$ with finite support which vanish on $X$. For such $u$, note that $\Omega_t(f)$ is contained in $\cV \backslash X$. Hence \eqref{eq:ineqproof} is valid with $\alpha (\cV \backslash X)$ instead of $\alpha$. Then $2 \lambda_0^{\ess} (\rh) \geq \alpha(\cV \backslash X)^2$ and the second estimate follows.
\end{proof}

\begin{remark}\label{rem:ourLaplace}
For the difference expression $\tau_\cG$ defined in Section~\ref{ss:III.00}, the function $m$ is given by \eqref{eq:m_def} and the edge weight $b$ is defined by $b(e):=1/|e|$ for all $e\in\cE$. Hence setting $d(e):=|e|$ for $e\in\cE$, we conclude that $|\cdot|$ is intrinsic in the sense of \eqref{eq:intr} since 
\[
\sum_{e \in \cE_v} d(e)^2 b(e) = \sum_{e \in \cE_v} |e|^2 \frac{1}{|e|} = \sum_{e \in \cE_v} |e| = m(v)
\] 
for all $v\in\cV$. Moreover, in this case we have
\[
(d\cdot b)(\cE_b(X))  =  \sum_{e \in \cE_b(X)} d(e)b(e) = \sum_{e \in \cE_b(X)} |e|\frac{1}{|e|} = \#(\cE_b(X)),
\]
and hence \eqref{eq:Iso} and \eqref{eq:IsoEss} coincide with \eqref{eq:discrCheeger} and \eqref{eq:discrCheegeress}, respectively. In particular, Theorem \ref{th:cheegerdiscr} implies the estimate \eqref{eq:Cheegerdiscr}.
\end{remark}

%%%%%%%%%%%%%%%%%%%%%%%%%%%%%%%%%%%%%%%%%%%%%%%%%%%%%%%%%%%%%%%%%
%%%%%%%%%%%%%%%%%%%%%%%%%%%%%%%%%%%%%%%%%%%%%%%%%%%%%%%%%%%%%%%%%
%%%%%%%%%%%%%%%%%%%%%%%%%%%%%%%%%%%%%%%%%%%%%%%%%%%%%%%%%%%%%%%%%

\noindent
\ack We thank Pavel Exner, Delio Mugnolo, Olaf Post and Wolfgang Woess for useful discussions and hints with respect to the literature.
We also thank the referee for the careful reading of our manuscript and hints with respect to the literature.

\end{document}